\documentclass{article}
\usepackage{amsmath, amssymb, amsthm, verbatim}

\newtheorem*{yo}{Yoneda lemma for formal group schemes}

\newtheorem{logo}{Proposition}[section]
\newtheorem{homo}[logo]{Proposition}
\newtheorem{hon}[logo]{Proposition}
\newtheorem{ass}[logo]{Lemma}
\newtheorem{glob}[logo]{Proposition}

\newtheorem{ufix_suf}{Proposition}[section]
\newtheorem{mor_ufix}[ufix_suf]{Proposition}
\newtheorem{ufix_prod}[ufix_suf]{Proposition}
\newtheorem{ker_zero}[ufix_suf]{Lemma}
\newtheorem{ufix_ex}{Theorem}
\newtheorem{glufix}[ufix_suf]{Proposition}

\newtheorem{spf}{Proposition}[section]
\newtheorem{hom_wr}[spf]{Proposition}
\newtheorem{log_wr}[spf]{Proposition}
\newtheorem{tran_wr}[spf]{Proposition}

\newtheorem{wr_sch}{Proposition}[section]
\newtheorem{mod_gr}[wr_sch]{Proposition}
\newtheorem{tor_mod}[wr_sch]{Proposition}
\newtheorem{fix_sch}[wr_sch]{Proposition}
\newtheorem{tam_iso}[wr_sch]{Proposition}
\newtheorem{cc_iso}[wr_sch]{Proposition}
\newtheorem{wr_mfg}[wr_sch]{Proposition}

\newtheorem{wr_mgs}{Proposition}[section]
\newtheorem*{wr_log}{Corollary}
\newtheorem{linco}[wr_mgs]{Proposition}
\newtheorem{q_mat}[ufix_ex]{Theorem}
\newtheorem*{q_mat+}{Corollary}
\newtheorem{repr}[ufix_ex]{Theorem}

\newtheorem{sm_ram}{Lemma}[section]
\newtheorem{tip3}[sm_ram]{Proposition}
\newtheorem{v_unr}[sm_ram]{Lemma}
\newtheorem{edlem}[sm_ram]{Lemma}
\newtheorem{eigen}[sm_ram]{Proposition}
\newtheorem{v_ram}[sm_ram]{Lemma}
\newtheorem{tpt3}[ufix_ex]{Theorem}
\newtheorem*{tpt3+}{Corollary}
\newtheorem{isrefg}[sm_ram]{Lemma}
\newtheorem{isreto}[sm_ram]{Proposition}

\newtheorem{tip41}{Proposition}[section]
\newtheorem{tpt4}[tip41]{Proposition}

\newtheorem{tip51}{Proposition}[section]
\newtheorem{ms_pabz}[tip51]{Lemma}
\newtheorem{tpt5}[ufix_ex]{Theorem}
\newtheorem{appl}[tip51]{Proposition}
\newtheorem*{isfoco}{Corollary}

\newcommand{\Z}{\mathbb{Z}}
\newcommand{\Q}{\mathbb{Q}}
\newcommand{\G}{\mathbb{G}}
\newcommand{\F}{\mathbb{F}}
\newcommand{\LL}{\mathbb{L}_m}

\newcommand{\id}{{\rm id}}
\newcommand{\Gal}{{\rm Gal}}
\newcommand{\Hom}{{\rm Hom}}
\newcommand{\Aut}{{\rm Aut}}

\newcommand{\Ima}{{\rm Im\ }}
\newcommand{\M}{{\rm M}}
\newcommand{\Sp}{{\rm Sp\ }}
\newcommand{\Spf}{{\rm Spf\ }}
\newcommand{\GL}{{\rm GL}}

\newcommand{\Ker}{{\rm Ker}}
\newcommand{\Kera}{{\rm Ker\ }}

\newcommand{\rk}{{\rm rk}}
\newcommand{\caa}{{\cal A}}
\newcommand{\cab}{{\cal B}}
\newcommand{\cac}{{\cal C}}
\newcommand{\cae}{{\cal E}}

\newcommand{\cag}{{\cal G}}
\newcommand{\cah}{{\cal H}}
\newcommand{\caj}{{\cal J}}
\newcommand{\cai}{{\cal I}}
\newcommand{\call}{{\cal L}}

\newcommand{\can}{{\cal N}}
\newcommand{\cao}{{\cal O}}
\newcommand{\car}{{\cal R}}
\newcommand{\cat}{{\cal T}}
\newcommand{\cau}{{\cal U}}
\newcommand{\cax}{{\cal X}}
\newcommand{\cay}{{\cal Y}}
\newcommand{\caz}{{\cal Z}}
\newcommand{\wh}{\widehat}
\newcommand{\bt}{\blacktriangle}

\renewcommand{\phi}{\varphi}
\renewcommand{\epsilon}{\varepsilon}

\begin{document}

\title{Formal completions of N\'eron models for algebraic tori}
\author{Oleg Demchenko\thanks{The author wishes to thank JSPS
for financial support and the Mathematical Institute of Tohoku
University for its hospitality.}, Alexander Gurevich\thanks
{The author was partially supported by the Minerva Foundation
through the Emmi Noether Institute, the Ministry of
Absorption, Israel, and the grant 3-3578 from the Ministry
of Science, Israel, and the Russian Foundation for Basic
Research.}, Xavier Xarles\thanks{The author was partially
supported by the grant MTM 2006-11391 from DGI.}}

\maketitle

\begin{abstract}
We calculate the formal group law which represents the completion
of the N\'eron model for an algebraic torus over $\Q$ split
in a tamely ramified abelian extension. To that end, we introduce
an analogue of the fixed part of a formal group law with respect
to a group action and give a method to compute its Honda type.
\end{abstract}



\section*{Introduction}

\bigskip
An explicit description of a formal group law which
represents the completion of the N\'eron model for an elliptic
curve over $\Q$ was given by Honda with the aid of the $L$-series
of the curve considered (see \cite{Ho1}). The first attempt to
construct such a formal group law for algebraic tori was made by
Deninger and Nart \cite{DN}. They consider a torus over $\Q$ which
is split over an abelian tamely ramified extension $K$ of $\Q$ and
define a formal group law in terms of the Galois representation
corresponding to this torus. Employing Honda's isomorphism
criterion for formal group laws (see \cite{Ho}) they show that
after inverting some primes, the formal group obtained becomes
isomorphic to the completion of the N\'eron model for the torus.

In our article, we work in the same setting as Deninger and Nart
and calculate a formal group law which represents the completion
of the N\'eron model. Our main tools are Edixhoven's
interpretation of the N\'eron model as a maximal fixed subscheme
\cite{Ed}, Honda's theory of formal group laws \cite{Ho} and a
theory of universal fixed pairs for a formal group law with a
group action which is introduced within the scope of this note.

The outline of the paper is as follows. In Section~1, we review
the notion of the formal completion of a group scheme and
introduce related notation. We proceed with the definition and
basic properties of formal group laws accompanied with the main
results of Honda's theory (Section~2). In Section~3, formal group
laws supplied with a group action are studied. We introduce
the notion of a universal fixed pair widely employed in subsequent
constructions and find universal fixed pairs for certain formal
group laws over the ring of integers in an unramfied extension of
$\Q_p$ (Theorem~\ref{ufix_ex}).
Moreover we establish a sufficient condition for a formal group law
over $\Z$ to appear in a universal fixed pair. The construction
of the Weil restriction of a formal group law is considered in
Section~4. 
We present an explicit expression for the logarithm of the Weil
restriction of the multiplicative formal group law and further
compute its Honda's type.
Another Honda's type of the same formal group law is computed
in \cite{Ib}. An advantage of our type is that its coefficients
belong to $\Z$, and thus it can be used for finding a universal
fixed pair in the global case.

Section~5 is devoted to the definition of N\'eron model and
its main properties. Our reference for this topic is \cite{BLR}.
We also recall Edixhoven's result (see \cite{Ed})
which asserts that if $T$ is a torus over $L$ and $K/L$ is a
tamely ramified Galois extension, then the N\'eron model for $T$
is isomorphic to the maximal fixed subscheme in the Weil
restriction of the N\'eron model for $T_K$ with respect to the
natural action of $\Gal(K/L)$.

In Section~6, we consider a $d$-dimensional torus $T$ over $L$
split over a Galois extension $K$ of $L$ and prove that the
cotangent space of the Weil restriction of the N\'eron model for
$T_K$ is isomorphic to
$$(\cax\otimes_\Z\cao_L)\otimes_{\cao_L}\Hom_{\cao_L}(\cao_K,\cao_L)$$
as $\Gal(K/L)$-module, where $\cax$ is the Galois module of
$K$-characters corresponding to $T$ (Theorem~\ref{q_mat}).
Further we show that if $K/L$ is tamely ramified, then
the Weil restriction $\Phi$ of the direct sum of $d$ copies of
the multiplicative formal group law provided with the Galois
action whose linear part is given in Theorem~\ref{q_mat} admits
a universal fixed pair.
Since the formal completion commutes with taking the maximal
fixed subscheme and Weil restriction, we conclude that the formal
group law which appears in this universal fixed pair represents
the completion of the N\'eron model for $T$ (Theorem~\ref{repr}).
Further we consider several special cases and apply the techniques
of Sections~3 and 4 in order to calculate the formal group law
representing the completion of the N\'eron model explicitly.

The case where $L=\Q_p$ and $K$ is an abelian tamely ramified
extension of $\Q_p$ is treated in Section 7. We prove that
the formal completion is isomorphic to the direct sum of
a $p$-divisible group whose dimension is equal to that of
the maximal subtorus of $T$ split over the maximal unramified
subextension of $K$, and several copies of the additive formal
group schemes (Theorem~\ref{tpt3}). If $K/\Q_p$ is unramified,
we get the same answer as in \cite{DN}.
In another particular case, when $K/\Q_p$ is totally ramified,
the formal completion turns out to be isomorphic to the direct
sum of several copies of the multiplicative and additive formal
group schemes. A similar result for the reduction
of the N\'eron model for such tori appears in \cite{NX}.
We apply Theorem~\ref{tpt3} to show that in the case considered,
the formal completion of the N\'eron model is uniquely determined
by its reduction, and give an example of two tori such that
their N\'eron models have isomorphic completions, but
non-isomorphic reductions.

Section~8 is devoted to one-dimensional tori over $L=\Q$ which
are split over a tamely ramified quadratic extension $K$ of $\Q$.
We construct two formal group laws representing the completion
of the N\'eron model (for one of them this was proven in \cite{DN})
and reprove Honda's theorem (see \cite{Ho1}) which asserts that
the formal group law obtained from the Dirichlet series of $K/\Q$
is strongly isomorphic to the formal group law
$F_q(x,y)=x+y+\sqrt q xy$, where $q$ is the discriminant of
$K/\Q$. The case where $L=\Q$ and $K$ is an abelian tamely
ramified extension of $\Q$, is considered in Section~9.
Due to the Kronecker-Weber theorem, one can suppose that
$K=\Q(\xi)$, where $\xi$ is a $q$-th primitive root of unity,
and $q$ is the product of distinct primes. We express
the coefficients of the corresponding formal group law
in terms of the images of the Frobenius automorphisms
with respect to the Galois representation in the character
group of $T$ (Theorem~\ref{tpt5}). This is our main result.
As an application, we show that a torus over $\Q$ split over
an abelian tamely ramified extension of $\Q$ is determined
up to isomorphism by the completion of its N\'eron model.
A similar result for Jacobians and elliptic curves was proven
in \cite{Na}.

Throughout the paper, we use the following matrix notation.\\
\noindent
If $U=\{a_{i,j}\}_{0\leq i\leq m-1;0\leq j\leq n-1}$,
$V=\{b_{i',j'}\}_{0\leq i'\leq m'-1;0\leq j'\leq n'-1}$
are matrices, their Kronecker product $U\otimes V$ is a matrix
$W=\{c_{k,l}\}_{0\leq k\leq mm'-1;0\leq l\leq nn'-1}$,
where $c_{i'm+i,j'n+j}=a_{i,j}b_{i',j'}$ for
$0\leq i\leq m-1,0\leq j\leq n-1,0\leq i'\leq m'-1,0\leq j'\leq n'-1$.
Notice that $(U\otimes V)(U'\otimes V')=(UU'\otimes VV')$
and $(U\otimes V)^T=U^T\otimes V^T$.

To make calculations with Kronecker product easier, we sometimes
employ the following ``matrix-of-matrices'' representation.
Let $R$ be a ring and let $U^{(i,j)}, 0\leq i,j\leq n-1$
be matrices from $\M_m(R)$ and
$U^{(i,j)}=\{a_{k,l}^{(i,j)}\}_{0\leq k,l\leq m-1}$.
Then the correspondence
$$
\{U^{(i,j)}\}\mapsto \{c_{s,t}\}_{0\leq s,t\leq mn-1},\quad
c_{im+k,jm+l}=a_{k,l}^{(i,j)}
$$
is a bijection between $\M_n(\M_m(R))$ and $\M_{mn}(R)$. Remark that
if $U\in\M_m(R)$ and $V=\{b_{i,j}\}_{0\leq i,j\leq n-1}\in M_n(R)$
then $\{Ub_{i,j}\}_{0\leq i,j\leq n-1}\mapsto U\otimes V$.

Denote $n\times n$ matrices
$$
J_{n}=
\begin{pmatrix}
1      & 0 & \cdots & 0\\
0      & 0 &        & 0\\
\vdots &   & \ddots & \vdots\\
0      & 0 & \cdots &  0
\end{pmatrix}
,\quad
J'_n=\begin{pmatrix}
1      & \cdots & 1\\
\vdots & \ddots & \vdots\\
1      & \cdots & 1
\end{pmatrix}
,\quad
P_{n}=
\begin{pmatrix}
0      & 0 & \cdots & 0 & 1\\
1      & 0 &        & 0 & 0\\
0      & 1 &        & 0 & 0\\
\vdots &   & \ddots &   & \vdots\\
0      & 0 & \cdots & 1 & 0
\end{pmatrix}.
$$
As usual, the $n\times n$ identity matrix is denoted by $I_n$,
and the $m\times n$ matrix
$\begin{pmatrix}
I_n\\
0
\end{pmatrix}$
is denoted by $I_{m,n}$ ($n\leq m$).
Finally, we write $\delta_i^j$ for Kronecker's delta, i.e.
$\delta_i^j=1$, if $i=j$, and $\delta_i^j=0$ otherwise.

\section{Completion of group schemes along the zero section}

Let $A$ be a ring. Let $\cac$ be a pro-nilpotent
augmentated $A$-algebra, i.e. an $A$-algebra with augmentation
map $\epsilon:\cac\to A$, $\epsilon(a)=a$ for any $a\in A$, and
decreasing chain of ideals $\caj_k$ such that $\Kera\epsilon=\caj_1$,
and $\caj_1/\caj_k$ is a nilpotent $A$-algebra for any $k$.
If $\cac$ is complete and Hausdorff in the topology defined
by $\caj_k$, one can define the functor $\Spf \cac$ from the
category of nilpotent $A$-algebras to the category of sets by
$\Spf\cac(N)=\varinjlim\Hom^*_A(\cac/\caj_k,A\oplus N)$, where
$\Hom^*$ is the set of homomorphisms of augmentated algebras.
A functor from the category of nilpotent $A$-algebras
to the category of sets is called a {\it formal scheme} over $A$,
if it is isomorphic to $\Spf\cac$ for some complete Hausdorff
pro-nilpotent augmentated $A$-algebra $\cac$. Evidently, any formal
scheme can be extended to a functor defined on the category of
complete Hausdorff pro-nilpotent augmentated $A$-algebras.
A continuous homomorphism $h:\cac'\to\cac$ of complete Hausdorff
pro-nilpotent augmentated $A$-algebras defines a morphism
$\Spf h:\Spf\cac\to\Spf\cac'$ in the following way: if
$s\in\Spf(\cac)(N)$, then $\Spf h(s)=s\circ h\in\Spf(\cac')(N)$.

\begin{yo} \label{yo}
Let $\cac$, $\cac'$ be complete Hausdorff pro-nilpotent
augmentated $A$-algebras, and $\eta:\Spf\cac\to\Spf\cac'$
be a morphism of functors. Then there is a unique continuous
homomorphism $h:\cac'\to\cac$ such that $\eta=\Spf h$.
\end{yo}

Let $\cah$ be an $A$-algebra and $\caj\subset \cah$ be an
ideal in $\cah$ such that $\cah/\caj\cong A$ and $\cap \caj^k=0$.
Then $\hat\cah_\caj=\varprojlim \cah/\caj^k$ with
the augmentation map $\hat\cah_\caj\to\hat\cah_\caj/\caj_1$
and the chain of ideals
$\caj_k=\Ker(\hat\cah_\caj\to\cah_\caj/\caj^k)$
is a complete Hausdorff pro-nilpotent
augmentated $A$-algebra. In this case, the formal $A$-scheme
$\Spf\hat \cah_\caj$ is called the {\it formal completion of}
$S=\Sp \cah$ {\it along} $\caj$ and is denoted by $\hat S_\caj$.
Remark that if $k$ is such that
$N^k=0$, then
$\Hom^*_A(\cah/\caj^k,A\oplus N)=\Hom^*_A(\cah,A\oplus N)$.
Hence $\hat S_\caj(N)$ can be identified with the subset of
$S(A\oplus N)=\Hom_A(\cah,A\oplus N)$ consisting of
the homomorphisms which map $\caj$ in $N$.

A functor from the category of nilpotent $A$-algebras to
the category of groups is called a {\it formal group scheme}
over $A$ if its composition with the forgetful functor is
a formal scheme over $A$. Due to Yoneda Lemma, giving a
formal group scheme structure to a formal scheme $\Spf \cac$
is equivalent to fixing comultiplication, counit and
coinverse in $\cac$ which satisfy usual group axioms.

Let $G$ be an affine group scheme over $A$ with Hopf algebra
$\cah$, and $\caj$ denote the augmentation ideal in $\cah$.
Then $\cah/\caj\cong A$ and $\cap \caj^k=0$. Thus one can consider
the formal completion $\hat G_\caj$ that we denote just by
$\hat G$. For a nilpotent $A$-algebra $N$, the comultiplication
in $\cah$ induces a group structure on $\hat G(N)$ what converts
$\hat G$ into a formal group scheme. Moreover any morphism
$\eta:G\to G'$ of affine group schemes induces a morphism
$\hat\eta:\hat G\to\hat G'$ of their formal completions.

\section{Logarithms of formal group laws and their types}

Let $A$ be a ring. We denote by $X$ and $Y$ the sets
of variables $x_1,\ldots,x_d$ and  $y_1,\ldots,y_d$,
respectively. A $d$-dimensional {\it formal group law}
over $A$ is a $d$-tuple of formal power series
$F\in A[[X,Y]]^d$ such that

i) $F(X,0)=0$;

ii) $F(X,F(Y,Z))=F(F(X,Y),Z)$;

iii) $F(X,Y)=F(Y,X)$.

\noindent Let $F$ and $F'$ be $d$- and $d'$-dimensional formal
group laws over $A$. A $d'$-tuple of formal power series
$f\in A[[X]]^{d'}$ is called a {\it homomorphism} from $F$ to
$F'$, if $f(0)=0$ and $f(F(X,Y))=F'(f(X),f(Y))$. The matrix
$D\in\M_{d',d}(A)$ such that $f(X)=DX\mod\deg2$ is called
the {\it linear coefficient} of $f$. Formal group laws are
called {\it strongly isomorphic}, if there exists an isomorphism
between them whose linear coefficient is the identity matrix.

If $G$ is a smooth affine group scheme over $A$
with Hopf algebra $\cah$ and augmentation ideal $\caj$,
then any set of elements $x_1,...,x_d\in \caj$ such that
$x_1+\caj^2,...,x_d+\caj^2$ form a free $A$-basis of
$\caj/\caj^2$, gives rise to an isomorphism between
$\hat\cah_\caj$ and $A[[X]]$ which provides a formal group
scheme structure to $\Spf A[[X]]$. The images of
$x_1,...,x_d\in A[[X]]$ with respect to the comultiplication
form a $d$-tuple of elements of $A[[X,Y]]$ which is a
$d$-dimensional formal group law over $A$. Thus, for example,
the element $x-1$ of the augmentation ideal in the Hopf
algebra $\Z[x,x']/(1-xx')$ of the multiplicative group
scheme $\G_m$ over $\Z$, induces the multiplicative
formal group law $\F_m(x,y)=x+y+xy$.

For a morphism $\eta\colon G\to G'$ of smooth affine group
schemes over $A$, denote by $C$ the matrix of the $A$-module
homomorphism from $\caj'/\caj'^2$ to $\caj/\caj^2$ in
the bases $x'_1+\caj'^2,\ldots,x'_{d'}+\caj'^2$ and
$x_1+\caj^2,\ldots,x_d+\caj^2$. Then the linear coefficient
of the formal group law homomorphism corresponding to
$\hat\eta$ in the same bases is $C^T$.

If $\lambda\in A[[X]]^d$ and $\lambda(X)\equiv X\mod\deg 2$,
then there exists a unique
inverse under composition $\lambda^{-1}$. 
In this case,
$
F_\lambda(X,Y)
=\lambda^{-1}(\lambda(X)+\lambda(Y))
$
is a $d$-dimensional formal group law over $A$, and
$\lambda\in\Hom_A(F_\lambda,(\F_a)^d_A)$, where $\F_a(x,y)=x+y$
is the additive formal group law over $\Z$.

Let $\lambda\in A[[X]]^d$, $\lambda(X)\equiv X\mod\deg 2$ and
$\lambda'\in A[[X']]^{d'}$, $\lambda'(X')\equiv X'\mod\deg 2$,
where $X'$ is the set of variables $x'_1,\ldots,x'_{d'}$.
If $D\in\M_{d',d}(A)$, then
$\lambda'^{-1}\circ D\lambda\in A[[X]]^{d'}$
is a homomorphism from $F_{\lambda}$ to $F_{\lambda'}$.

\begin{logo}\cite[Theorem 1]{Ho}\label{logo}
For any $d$-dimensional formal group law $F$ over $\Q$-algebra $\caa$,
there exists a unique
$\lambda\in\Hom_{\caa}(F,(\F_a)^d_{\caa})$
such that $\lambda(X)\equiv X\mod\deg 2$.
\end{logo}

Let $A$ be a ring of characteristic 0. Then $\caa=A\otimes_{\Z}\Q$
is a $\Q$-algebra. If $F$ is a formal group law over $A$, then
applying Proposition~\ref{logo} to $F_\caa$ we obtain
$\lambda\in\caa[[X]]^d$ which is called the {\it logarithm}
of $F$. The logarithm of $\F_m$ is
$\LL (x)=\sum_{i=1}^\infty(-1)^{i+1}x^i/i$.

\begin{homo}\cite[Proposition 1.6]{Ho}\label{homo}
Let $F$ and $F'$ be $d$- and $d'$-dimensional formal group laws
over $A$ with logarithms $\lambda$ and $\lambda'$, respectively,
and $f\in\Hom_A(F,F')$, then $f=\lambda'^{-1}\circ D\lambda$,
where $D\in\M_{d',d}(A)$ is the linear coefficient of $f$.
\end{homo}

Let $L$ be a finite unramified extension of $\Q_p$ with integer ring
$\cao_L$ and Frobenius automorphism $\sigma$.
Denote $\call_d=L[[X]]$
and let $\bt:\call_d\to\call_d$ be a $\Q_p$-algebra map
defined by  $\bt(x_i)=x_i^p$ and $\bt(a)=\sigma(a)$,
where $a\in L$.
Let $\cae=\cao_L[[\bt]]$ be a noncommutative $\Q_p$-algebra with
multiplication rule $\bt a=\sigma(a)\bt$, $a\in \cao_L$.
Then $\call_d$ has a left $\cae$-module structure which
induces a left $\M_d(\cae)$-module structure on $\call_d^d$.

If $u\in\M_d(\cae)$, $u\equiv pI_d\mod\bt$ and $\lambda\in\call_d^d$
are such that $u\lambda\equiv 0\mod p$, we say that $\lambda$
is {\it of type} $u$. Clearly, if $u\in\M_d(\cae)$, $u\equiv pI_d\mod\bt$,
then $(u^{-1}p)(\id)\in\call_d^d$ is of type $u$, and
$((u^{-1}p)(\id))\equiv X\mod\deg2$. Remark also that
$\LL $ is of type $p-\bt$.

\begin{hon}\cite[Theorem 2, Proposition 3.3, Theorem 3]{Ho}\label{hon}
\begin{description}
\item[(i)] If $\lambda\in\call_d^d$ is of type $u$, then
$\lambda$ is the logarithm of a formal group law over $\cao_L$.
\item[(ii)] For any formal group law $F$ over $\cao_L$ with logarithm
$\lambda\in\call_d^d$ there exists $u\in\M_d(\cae)$ such that $\lambda$
is of type $u$.
\item[(iii)] Let $F, F'$ be $d$- and $d'$-dimensional formal
group laws over $\cao_L$ with logarithms $\lambda, \lambda'$
of type $u, u'$, respectively, and $D\in\M_{d',d}(\cao_L)$. Then
$\lambda'^{-1}\circ D\lambda\in\Hom_{\cao_L}(F,F')$ if and only if
there exists $w\in\M_{d',d}(\cae)$ such that $u'D=wu$.
\end{description}
\end{hon}

We will also use the following technical result.

\begin{ass}\cite[Lemma 2.3]{Ho}\label{ass}
If $v\in\cae$, $\lambda\in\call_1$ is of type $u\in\cae$ and
$\phi\in O_L[[X]]$, then
$v(\lambda\circ\phi)\equiv(v\lambda)\circ\phi\mod p$.
\end{ass}

If $\lambda\in\Q[[X]]^d$ and  $p$ is a prime,
we call a type of $\lambda$ considered as an element of
$\Q_p[[X]]^d$ its {\it $p$-type}. To avoid confusion
in this case, instead of $\bt$ and $\cae$, we write
$\bt_p$ and $\cae_p$.

Let $\Xi$ be a map from the set of prime numbers to $\M_d(\Z)$
such that any two matrices from its image commute.
If $p_1,\ldots,p_k$ are distinct primes and $m=\prod_{i=1}^kp_i^{t_i}$,
put $A_m=\prod_{i=1}^k \Xi(p_i)^{t_i}$ and
$\lambda_\Xi=\sum_{m=1}^\infty A_mX^m/m\in\Q[[X]]$. Finally,
define $F_\Xi(X,Y)=\lambda_\Xi^{-1}(\lambda_\Xi(X)+\lambda_\Xi(Y))$.

\begin{glob}\cite[Theorem 8]{Ho}\label{glob}
$\lambda_\Xi$ is of $p$-type $pI_d-\Xi(p)\bt_p\in\M_d(\cae_p)$,
and $F_\Xi$ is a formal group law over $\Z$.
\end{glob}

\begin{proof}
For any prime $p$, we have
$$
(pI_d-\Xi(p)\bt_p)\lambda_\Xi=(pI_d-A_p\bt_p)\left(\sum_{(m,p)=1}
\sum_{i=0}^\infty\frac{A_mA_p^i}{mp^i}X^{mp^i}\right)
$$
$$
=\sum_{(m,p)=1}\frac{pA_m}{m}X^m\equiv 0\mod p.
$$
Thus $\lambda_\Xi$ is of the required $p$-type.
By Proposition~\ref{hon} (i), $F_\Xi$ is a formal group law over $\Z_p$.
Since it is true for any prime $p$, it is defined over $\Z$.
\end{proof}

\section{ Universal fixed pairs for group actions on formal group laws}

The action of a group $\cag$ on a formal group law $\Phi$ over
a ring $A$ is given by a homomorphism $\cag\to\Aut_A(\Phi)$.
We denote the image of $\sigma\in\cag$ under this homomorphism
also by $\sigma$. It will never lead to confusion.

Let $F$ be a formal group law over $A$ and
$f\in\Hom_A(F,\Phi)$. A pair $(F,f)$ is called {\it fixed} for
$(\Phi,\cag)$ if $\sigma\circ f=f$ for any $\sigma\in\cag$. Yoneda Lemma
implies that a pair $(F,f)$ is fixed if and only if for any
nilpotent $A$-algebra $N$, we have $\Ima f(N)\subset \Phi(N)^{\cag}$.

A fixed pair $(F,f)$ is called {\it universal} if for any fixed
pair $(F',f')$, there exists a unique
$g\in\Hom_A(F',F)$ such that $f'=f\circ g$.
Clearly, if $(F,f)$ and $(\tilde F,\tilde f)$ are
universal fixed pairs, then $g\in\Hom_A(\tilde F,F)$
satisfying $\tilde f=f\circ g$ is an isomorphism.

\begin{ufix_suf} \label{ufix_suf}
If $(F,f)$ is a fixed pair such that for any
nilpotent $A$-algebra $N$ the map $f(N): F(N)\to \Phi(N)^{\cag}$ is
bijective, then it is universal.
\end{ufix_suf}

\begin{proof}
We notice that for any nilpotent $A$-algebra $N$ and
any fixed pair $(F',f')$, there exists a unique homomorphism
$ g(N):F'(N)\to F(N)$ such that $f(N)\circ g(N)=f'(N)$,
and that $ g(N)$ is functorial in $N$. Then the statement follows
from Yoneda Lemma.
\end{proof}



\begin{mor_ufix} \label{mor_ufix}
Let $\cag$ act on formal group laws $\Phi$ and $\Phi'$,
and $\phi\in\Hom_A(\Phi,\Phi')$
commute with the actions of $\cag$. If $(F,f)$ and
$(F',f')$ are universal fixed pairs for $(\Phi,\cag)$
and $(\Phi',\cag)$, respectively, then there exists
a unique $\tilde\phi\in\Hom_A(F,F')$ such that
$\phi\circ f=f'\circ\tilde\phi$.
\end{mor_ufix}

\begin{proof}
For any $\sigma\in\cag$, we have
$\sigma\circ\phi\circ f=\phi\circ\sigma\circ f=\phi\circ f$.
Hence, $(F,\phi\circ f)$ is a fixed pair for $(\Phi',\cag)$.
Therefore, there exists a unique $\tilde\phi\in\Hom_A(F,F')$
such that $\phi\circ f=f'\circ\tilde\phi$.
\end{proof}

Let $\cag_1,\cag_2$ act on $\Phi$ and $\sigma_1\sigma_2=\sigma_2\sigma_1$
for any $\sigma_1\in\cag_1,\sigma_2\in\cag_2$.
If $(F_1,f_1)$ is a universal fixed pair for $(\Phi,\cag_1)$
then for any $\sigma_2\in\cag_2$ we have
$\sigma_1\circ(\sigma_2\circ f_1)=\sigma_2\circ(\sigma_1\circ
f_1)=\sigma_2\circ f_1$ for any $\sigma_1\in\cag_1$.
Thus there exists a unique $\sigma'_2\in\Aut_A(F_1)$ such that
$f_1\circ\sigma'_2=\sigma_2\circ f_1$. It induces an action of $\cag_2$
on $F_1$.

\begin{ufix_prod} \label{ufix_prod}
Let $\cag=\cag_1\times\cag_2$ act on a formal group law $\Phi$.
If $(F_1,f_1)$ is a universal fixed pair for $(\Phi,\cag_1)$ and
$(F_2,f_2)$ is a universal fixed pair for $(F_1,\cag_2)$ with respect to
the induced action of $\cag_2$ on $F_1$,
then $(F_2,f_1\circ f_2)$ is a universal fixed pair for $(\Phi,\cag)$.
\end{ufix_prod}

\begin{proof}
Let $(F',f')$ be a fixed pair for $(\Phi,\cag)$, i.e.,
$\sigma\circ f'=f'$ for any $\sigma\in\cag$.
It implies that $(F',f')$ is a fixed pair for $(\Phi,\cag_1)$
thus giving $g_1\in\Hom_A(F',F_1)$ such that $f'=f_1\circ g_1$.
Then for any $\sigma_2\in\cag_2$, we have
$f_1\circ(\sigma'_2\circ g_1)=\sigma_2\circ f_1\circ g_1=
\sigma_2\circ f'=f'=f_1\circ g_1$ which implies
$\sigma'_2\circ g_1=g_1$ by universality of $(F_1,f_1)$.
It gives $g_2\in\Hom_A(F',F_2)$ such that $g_1=f_2\circ g_2$.
Thus $(f_1\circ f_2)\circ g_2=f_1\circ g_1=f'$.

The uniqueness of $g_2$ can be easily checked in a similar way.
\end{proof}

Let $L$ be a finite unramified extension of $\Q_p$ with integer ring
$\cao_L$ and residue field $k$.

\begin{ker_zero} \label{ker_zero}
Let $\cag'$ be a set of matrices over $\cao_L$ with
the same number of columns $n$ such that
$\cap_{D'\in\cag'}\Ker\left(D'\otimes k\right)=\{0\}$.
Then there exist $D'_i\in\cag'$, $D'_i\in\M_{n_i,n}(\cao_L)$
and $C_i\in\M_{n,n_i}(\cao_L)$ for $1\le i\le m$ such that
$\sum_{i=1}^m C_iD'_i\in\GL_n(\cao_L)$.
\end{ker_zero}

\begin{proof}
It is enough to prove that if $\overline\cag$ is a set of
matrices over $k$ with the same number of columns $n$ such that
$\cap_{\overline D\in\overline\cag}\Kera \overline D=\{0\}$,
there exist $\overline D_i\in\overline\cag$,
$\overline D_i\in\M_{n_i,n}(k), 1\le i\le m$, and
$\overline C_i\in\M_{n,n_i}(k), 1\le i\le m$, such that
$\sum_{i=1}^m \overline C_i\overline D_i\in\GL_n(k)$.
Consider a finite set of matrices
$\overline D_1,\ldots,\overline D_m\in\overline\cag$
whose kernels have zero intersection, and construct
a matrix $D^*$ with the rows of these matrices. Since
$\cap_{i=1}^m\Kera \overline D_i=\{0\}$, the rank of $D^*$
is equal to $n$. Then there is a matrix $C^*$
with $n$ rows such that $C^*D^*=I_n$.
The required matrices $\overline C_1,\ldots,\overline C_m$
are formed by the corresponding columns of $C^*$.
\end{proof}

\begin{ufix_ex} \label{ufix_ex}
{\rm I.} Let $\cag^*\subset\M_r(\cao_L)$. The following
conditions are equivalent\\
{\rm (i)} $\rk_{\cao_L}\cap_{D\in\cag^*}\Kera D=
\dim_k\cap_{D\in\cag^*}\Ker\left(D\otimes k\right)$;\\
{\rm (ii)} If $x\in\cao_L^r$ is such that
for any $D\in\cag^*$ there is $y_D\in\cao_L^r$
satisfying $D x=py_D$,
then there exists $x'\in\cao_L^r$ such that
$D x'=y_D$ for any $D \in\cag^*$;\\
{\rm (iii)} There exist $0\le e\le r$ and $Q\in\GL_r(\cao_L)$
such that for any $D \in\cag^*$
$$
Q^{-1}D Q=
\begin{pmatrix}
0 & \hat D \\
0 & \tilde D
\end{pmatrix}, \qquad
\hat D \in\M_{e,r-e}(\cao_L),\,
\tilde  D \in\M_{r-e}(\cao_L),
$$
and there are $D_i\in\cag^*$, $\hat C_i\in\M_{r-e,e}(\cao_L)$
and $\tilde C_i\in\M_{r-e}(\cao_L)$, $1\leq i\leq m$, such that
$\sum_{i=1}^m\hat C_i\hat D_i+\tilde C_i\tilde D_i\in\GL_{r-e}(\cao_L)$.\\

{\rm II.} Let $\Phi $ be an $r$-dimensional formal group law over
$\cao_L$ with logarithm $\Lambda$. Let a group $\cag$ act on $\Phi$ and
$\cag^*=\{D\in\M_r(\cao_L):\Lambda^{-1}\circ (D+I_r)\Lambda\in\cag \}$.
If $\cag^*$ satisfies one of the above equivalent conditions
then $(\Phi ,\cag)$ has a universal fixed pair.
Moreover, if $\Lambda$ is of type $v$, then a universal fixed
pair $(F,f)$ can be taken in such a way that the linear coefficient
of $f$ is $QI_{r,e}$, and the logarithm of $F$ is of type $u$,
where $u$ is the upper-left $e\times e$-submatrix of $\tilde v=Q^{-1}vQ$.\\

\end{ufix_ex}

\begin{proof}
I. Remark that $\cap_{D\in\cag^*}\Kera D$
is a $p$-divisible $\cao_L$-module. Since the reduction of
$\cap_{D\in\cag^*}\Kera D$ is a subspace of
$\cap_{D\in\cag^*}\Ker\left(D\otimes k\right)$, condition (i)
is equivalent to the fact that they coincide. Obviously,
the latter is equivalent to condition (ii).

Now suppose that for any $D \in\cag^*$
$$
\begin{pmatrix}
0 & \hat D \\
0 & \tilde D
\end{pmatrix}
\begin{pmatrix}
\hat x\\
\tilde x
\end{pmatrix}=
\begin{pmatrix}
p\hat y_D\\
p\tilde y_D
\end{pmatrix}, \qquad
\hat x,\hat y_D\in\cao_L^e,\,
\tilde x,\tilde y_D\in\cao_L^{r-e}.
$$
Then $\hat D \tilde x=p\hat y_D$,
$\tilde D \tilde x=p\tilde y_D$ and for
$\tilde C_i, \hat C_i$ provided by condition (iii) we have
$(\sum_{i=1}^m\hat C_i\hat D_i+\tilde C_i\tilde D_i)\tilde x=
p\sum_{i=1}^m\hat C_i\hat y_{D_i}+\tilde C_i\tilde y_{D_i}$
which implies that $\tilde x=p\tilde x'$ for some
$\tilde x'\in\cao_L^{r-e}$. It gives for any $D \in\cag^*$
$$
\begin{pmatrix}
0 & \hat D \\
0 & \tilde D
\end{pmatrix}
\begin{pmatrix}
0\\
\tilde x'
\end{pmatrix}=
\begin{pmatrix}
\hat y_D\\
\tilde y_D
\end{pmatrix}
$$
as required.

For the reverse implication, take
$e=\rk_{\cao_L}\cap_{D\in\cag^*}\Kera D$ and let
$Q$ be the transition matrix from a basis of $\cao_L^r$
with first $e$ vectors from  $\cap_{D\in\cag^*}\Kera D$
to the standard basis.
Then for any $D \in\cag^*$, we have
$$
Q^{-1}D Q=
\begin{pmatrix}
0 & \hat D \\
0 & \tilde D
\end{pmatrix}, \qquad
\hat D \in\M_{e,r-e}(\cao_L),\,
\tilde  D \in\M_{r-e}(\cao_L).
$$
Denote
$\cag'=\{\tilde D \in\M_{r-e}(\cao_L)\colon D\in\cag^*\}\cup
\{\hat D \in\M_{e,r-e}(\cao_L)\colon D\in\cag^*\}$.

If $\tilde x\in k^{r-e}$ and
$\tilde x\in\cap_{D'\in\cag'}\Ker\left(D'\otimes k\right)$ then for any
$D \in\cag^*$
$$
\begin{pmatrix}
0 & \hat D \\
0 & \tilde D
\end{pmatrix}
\begin{pmatrix}
0\\
\tilde x
\end{pmatrix}=
\begin{pmatrix}
0\\
0
\end{pmatrix}
$$
which implies that
$$
\overline x=Q
\begin{pmatrix}
0\\
\tilde x
\end{pmatrix}
\in\cap_{D\in\cag^*}\Ker\left(D\otimes k\right).
$$
Then $\overline x$ is the reduction of an element from
$\cap_{D\in\cag^*}\Kera D$ by condition (i). It gives $\tilde x=0$,
i.e., $\cap_{D'\in\cag'}\Ker\left(D'\otimes k\right)=\{0\}$.
It remains to apply Lemma~\ref{ker_zero}.

II. If $\phi(X)=QX\in\call_r^r$ then
$\tilde\Lambda =\phi^{-1}\circ\Lambda\circ\phi$
is of type $\tilde v $. By Proposition~\ref{hon}~(i), $\tilde\Lambda $
is the logarithm of a formal group law $\tilde\Phi $ over $\cao_L$,
and $\phi\in\Hom_{\cao_L}(\tilde\Phi ,\Phi)$. Furthermore,
$\phi^{-1}\circ\sigma\circ\phi\in\Aut_{\cao_L}\tilde\Phi $
for any $\sigma\in\cag$ which defines an action of $\cag$ on
$\tilde\Phi $. If $D+I_r$ is the linear coefficient of $\sigma$, then
the linear coefficient of
$\phi^{-1}\circ\sigma\circ\phi$ is equal to $(Q^{-1}DQ+I_r)$.
Therefore Proposition~\ref{hon}~(iii) implies
the existence of $w \in\M_r(\cae)$
such that $\tilde v Q^{-1}D Q=w \tilde v $.
For $1\le i \le m$, let $w_i$ correspond to $D_i$ and
$$
\tilde v =\begin{pmatrix}
u  & * \\
\tilde u   & *
\end{pmatrix}, \quad
w_i=
\begin{pmatrix}
\hat z_i & \hat w_i \\
\tilde z_i  & \tilde w_i
\end{pmatrix},
$$
where $u,\hat z_i \in\M_e(\cae)$,
$\tilde u , \tilde z_i \in\M_{r-e,e}(\cae)$,
$\hat w_i\in\M_{e,r-e}(\cae)$,
$\tilde w_i \in\M_{r-e}(\cae)$.
Then
$$
\sum_{i=1}^m
\begin{pmatrix}
\hat C_i & \tilde C_i
\end{pmatrix}
w_i\tilde v =
\sum_{i=1}^m
\begin{pmatrix}
\hat C_i & \tilde C_i
\end{pmatrix}
\begin{pmatrix}
\hat z_i & \hat w_i \\
\tilde z_i  & \tilde w_i
\end{pmatrix}
\begin{pmatrix}
u  & * \\
\tilde u   & *
\end{pmatrix}
$$
$$
=\sum_{i=1}^m
\begin{pmatrix}
\hat C_i & \tilde C_i
\end{pmatrix}
\begin{pmatrix}
u  & * \\
\tilde u   & *
\end{pmatrix}
\begin{pmatrix}
0 & \hat D_i \\
0  & \tilde D_i
\end{pmatrix}=
\begin{pmatrix}
0 & *
\end{pmatrix}.
$$
It gives
$\sum_{i=1}^m(\hat C_i\hat z_i+\tilde C_i\tilde z_i)u+
\sum_{i=1}^m(\hat C_i\hat w_i+\tilde C_i\tilde w_i)\tilde u =0$.
Since $\hat w_i\equiv\hat D_i\mod\bt$ and
$\tilde w_i\equiv\tilde D_i\mod\bt$, we obtain that
$\sum_{i=1}^m\hat C_i\hat w_i+\tilde C_i\tilde w_i\in\M_{r-e}(\cae)$
is invertible, and hence, $\tilde u =zu$ for some $z\in\M_{r-e,e}(\cae)$.

Let $\lambda\in\call_e^e$ be of type $u$ and
$\lambda(X)\equiv X\mod\deg2$. Then by Proposition~\ref{hon} (i),
$\lambda$ is the logarithm of a formal group law $F$ over $\cao_L$.
Since
$$
\tilde v I_{r,e}=
\begin{pmatrix}
u  \\
\tilde u
\end{pmatrix}=
\begin{pmatrix}
I_e \\
z
\end{pmatrix}u,
$$
Proposition~\ref{hon}~(iii) implies that
$\tilde f =\tilde\Lambda ^{-1}\circ I_{r,e}\lambda
\in\Hom_{\cao_L}(F,\tilde\Phi )$.
Moreover, $(Q^{-1}DQ+I_r)I_{r,e}=I_{r,e}$ for any $D\in\cag^*$,
and hence, by Proposition~\ref{homo}, we get $
(\phi^{-1}\circ\sigma\circ\phi)\circ\tilde f=\tilde f$
for any $\sigma\in\cag$.
Thus $(F,\tilde f)$ is a fixed pair for $(\tilde\Phi,\cag)$.

Now let $(F',f')$ be another fixed pair for $(\tilde\Phi,\cag)$
and the linear coefficient of $f'$ be
$$
Z =
\begin{pmatrix}
\tilde Z\\
\hat Z
\end{pmatrix}
\in\M_{r,e'}(\cao_L),
\qquad {\rm where}\quad\tilde Z\in\M_{e,e'}(\cao_L),\,
\hat Z\in\M_{r-e,e'}(\cao_L).
$$
Then $(Q^{-1}DQ+I_r)Z =Z $, and therefore,
$\hat D_i\hat Z=0$ and $\tilde D_i\hat Z=0$, $1\leq i\leq m$.
Since $\sum_{i=1}^m\hat C_i\hat D_i+\tilde C_i\tilde D_i$
is invertible, $\hat Z=0$, and hence, $Z=I_{r,e}\tilde Z$.
According to Proposition~\ref{hon} (ii),
there exists $u'\in\M_{e'}(\cae)$ such that
the logarithm $\lambda'$ of $F'$ is of type $u'$.
By Proposition~\ref{hon} (iii), we have $\tilde v Z =w'u'$
for some
$$
w'=
\begin{pmatrix}
\tilde w' \\
*
\end{pmatrix}\in\M_{r,e'}(\cae),
\qquad {\rm where}\quad\tilde w'\in\M_{e,e'}(\cae).
$$
Then $u\tilde Z=\tilde w'u'$,
and Proposition~\ref{hon}~(iii) implies that
$g=\lambda^{-1}\circ\tilde Z\lambda'\in\Hom_{\cao_L}(F',F)$.
Besides, by Proposition~\ref{homo}, we have $f'=\tilde f \circ g$.

If $g'\in\Hom_{\cao_L}(F',F)$ is such
that $f'=\tilde f \circ g'$, and $Z'\in M_{e,e'}(\cao_L)$ is
the linear coefficient of $g'$, then
$I_{r,e}Z'=Z =I_{r,e}\tilde Z$, and hence, $Z'=\tilde Z$.
Then Proposition~\ref{homo} implies that $g'= g$, and thus
$(F,\tilde f )$ is a universal fixed pair for $(\tilde\Phi ,\cag)$.

Finally, let $f=\phi\circ \tilde f $. Then $f\in\Hom_{\cao_L}(F,\Phi)$,
the linear coefficient of $f$ is $QI_{r,e}$, and $(F,f)$ is a
universal fixed pair for $(\Phi,\cag)$.

\end{proof}

\begin{glufix}\label{glufix}
Let $\Phi$ be a formal group law over $\Z$ provided with an action
of a group $\cag$, and for any prime $p$, the logarithm of $\Phi$
be of $p$-type $v_p$. If there exist $Q\in\GL_r(\Z)$ and a formal
group law $F$ over $\Z$ such that for any prime $p$, $Q$ satisfies
condition {\rm (iii)} of Theorem~\ref{ufix_ex},{\rm I} for $L=\Q_p$,
and the logarithm of $F$ is of $p$-type $u_p$, where $u_p$ is the
upper-left $e\times e$-submatrix of $Q^{-1}v_pQ$, then there exists
$f\in\Hom_{\Z}(F,\Phi)$ such that the linear coefficient of $f$ is
equal to $QI_{r,e}$ and $(F,f)$ is a universal fixed pair for
$(\Phi,\cag)$.
\end{glufix}

\begin{proof}
By Theorem~\ref{ufix_ex} for any prime $p$, there exists
a universal fixed pair $(F_p,f_p)$ for $(\Phi,\cag)$ such that
the logarithm $\lambda_p$ of $F_p$ is of type $u_p$, and the linear
coefficient of $f_p$ is $QI_{r,e}$. Proposition~\ref{hon} (iii) implies
that $v_pQI_{r,e}=w_pu_p$ for some $w_p\in\M_{r,e}(\cae_p)$, and then
$f=\Lambda^{-1}\circ QI_{r,e}\lambda\in\Hom_{\Z_p}(F,\Phi)$, where
$\Lambda$ and $\lambda$ denote the logarithms of $\Phi$ and $F$,
respectively. Since it is true for any prime $p$,
we have $f\in\Hom_{\Z}(F,\Phi)$. Moreover by Proposition~\ref{homo},
we get $f=f_p\circ(\lambda_p^{-1}\circ\lambda)$, and hence,
$\sigma\circ f=\sigma\circ f_p\circ(\lambda_p^{-1}\circ\lambda)=f$
for any $\sigma\in\cag$. Thus $(F,f)$ is a fixed pair for
$(\Phi,\cag)$.

Let $(F',f')$ be another fixed pair for $(\Phi,\cag)$.
For every prime $p$, there exists $g_p\in\Hom_{\Z_p}(F',F_p)$
such that $f'=f_p\circ g_p$. Denote the linear coefficients of
$f'$ and $g_p$ by $Z\in\M_{r,e'}(\Z)$ and $Z_p\in\M_{e,e'}(\Z_p)$,
respectively, and let
$$
Q^{-1}Z=
\begin{pmatrix}
\tilde Z\\
\hat Z
\end{pmatrix}
\in\M_{r,e'}(\Z),
\qquad {\rm where}\quad\tilde Z\in\M_{e,e'}(\Z),\,
\hat Z\in\M_{r-e,e'}(\Z).
$$
Then $Z=QI_{r,e}Z_p$ implies $\hat Z=0$ and $\tilde Z=Z_p$,
in particular, the entries of $Z_p$ are in $\Z$.
According to Proposition~\ref{hon} (ii), for every prime $p$,
there exists $u'_p\in\M_{e'}(\cae_p)$ such that the logarithm
$\lambda'$ of $F'$ is of type $u'_p$. By Proposition~\ref{hon} (iii),
we have $u_p\tilde Z=w'_pu'_p$ for some $w'_p\in\M_{e,e'}(\cae_p)$,
and then $g=\lambda^{-1}\circ\tilde Z\lambda'\in\Hom_{\Z_p}(F',F)$.
Since it is true for any prime $p$, we get $g\in\Hom_{\Z}(F',F)$.
Besides, by Proposition~\ref{homo}, we have $f'=f\circ g$.

If $g'\in\Hom_{\Z}(F',F)$ is such that $f'=f\circ g'$, and
$Z'\in\M_{e,e'}(\Z)$ is the linear coefficient of $g'$, then
$QI_{r,e}Z'=Z=QI_{r,e}\tilde Z$, and hence, $Z'=\tilde Z$. Then
Proposition~\ref{homo} implies that $g'=g$, and thus
$(F,f)$ is a universal fixed pair for $(\Phi,\cag)$.
\end{proof}

\section{Weil restriction for formal group laws}

Let $A$ be a ring, $B$ be an $A$-algebra which is a free $A$-module
of finite rank with basis $e_0,\ldots,e_{n-1}$
that we fix throughout this section. For a positive
integer $d$, denote $B_d=B[[x_1,\ldots,x_d]]$. Define
the {\it Weil restriction} functor $\car_{B/A}(\Spf B_d)$ from
the category of nilpotent $A$-algebras to the category of sets
by $\car_{B/A}(\Spf B_d)(N)=\Spf B_d(N\otimes_AB)$.
Denote $A_d=A[[z_1,\ldots,z_{nd}]]$.
For any nilpotent $A$-algebra $N$, the map
$$\rho_d(N):\Spf A_d(N)\to \car_{B/A}(\Spf B_d)(N)$$
defined by
$\Bigl(\rho_d(N)(s)\Bigr)(x_l)=\sum_{j=0}^{n-1}s(z_{jd+l})\otimes e_j$,
$1\leq l\leq d$, is a bijection. Therefore, it gives an isomorphism
$\rho_d:\Spf A_d\to \car_{B/A}(\Spf B_d)$ that depends on
the fixed basis. It implies, in particular, that
$\car_{B/A}(\Spf B_d)$ is a formal scheme over $A$.
Let $f: B_{d'}\to B_d$ be a continuous homomorphism
of $B$-algebras. Denote by $\car_{B/A}(f): A_{d'}\to A_{d}$
the unique continuous homomorphism of $A$-algebras such that
$$
\sum_{j=0}^{n-1}\car_{B/A}(f)_{jd'+l}(z_1,\ldots,z_{nd})e_j=
f_l\left(\sum_{i=0}^{n-1}z_{id+1}e_i,\ldots,
\sum_{i=0}^{n-1}z_{id+d}e_i\right)\in B\otimes A_d
$$
for any $1\leq l\leq d'$, where $f_m\in B_{d}$ and
$\car_{B/A}(f)_m\in A_d$ are the images of $z_m$ with respect
to $f$ and $\car_{B/A}(f)$, respectively.
Certainly, $\car_{B/A}(f)$ also depends on the chosen basis.

\begin{spf}\label{spf}
$$
(\Spf f)(\rho_{d})=\rho_{d'}(\Spf \car_{B/A}(f)).
$$
\end{spf}

\begin{proof} Let $N$ be a nilpotent $A$-algebra and
$s\in\Hom_A(A_{d},N)$. Then
$$
\biggl((\Spf f)\Bigl(\rho_{d}(N)(s)\Bigr)\biggr)(x_l)
=\Bigl(\rho_{d}(N)(s)\Bigr)(f(x_l))
$$
$$
=f_l\Bigl(\rho_{d}(N)(s)(x_1),\dots,\rho_{d}(N)(s)(x_{d})\Bigr)
$$
$$
=f_l\left(\sum_{i=0}^{n-1}s(z_{id+1})\otimes e_i,\ldots,
\sum_{i=0}^{n-1}s(z_{id+d})\otimes e_i\right).
$$
for any $1\leq l\leq d'$. On the other hand,
$$
\biggl(\rho_{d'}(N)\Bigl((\Spf \car_{B/A}(f))(s)\Bigr)\biggr)(x_l)=
\sum_{j=0}^{n-1}\Bigl((\Spf \car_{B/A}(f))(s)\Bigr)(z_{jd'+l})\otimes e_j
$$
$$
=\sum_{j=0}^{n-1}s\Bigl(\car_{B/A}(f)(z_{jd'+l})\Bigr)\otimes e_j=
\sum_{j=0}^{n-1}\car_{B/A}(f)_{jd'+l}\Bigl(s(z_1),\ldots, s(z_{nd})\Bigr)e_j.
$$
\end{proof}

Let $F$ be a $d$-dimensional formal group law over $B$.
Then $F$ provides a formal group scheme structure to
$\Spf B_d$, and hence, also to $\car_{B/A}(\Spf B_d)$.
Further, $\rho_d$ allows to define a formal group scheme structure
on $\Spf A_d$, which gives an $nd$-dimensional formal group law
depending on the chosen basis. We denote this formal group law by
$\car_{B/A}(F)$.

\begin{hom_wr}\label{hom_wr}
If $F$ and $F'$ are formal group laws over $B$, and
$f\in\Hom_B(F,F')$, then
$\car_{B/A}(f)\in\Hom_A(\car_{B/A}(F),\car_{B/A}(F'))$.
\end{hom_wr}

\begin{proof}
Denote by $d,d'$ the dimensions of $F,F'$, respectively.
Let $N$ be a nilpotent $A$-algebra and
$s_1,s_2\in\Hom_A(A_d,N)$. Then Proposition~\ref{spf} yields
$$\rho_{d'}(N)\Bigl(\car_{B/A}(f)\bigl(\car_{B/A}(F)(s_1,s_2)\bigr)\Bigr)
=f\Bigl(F\bigl(\rho_d(N)(s_1),\rho_d(N)(s_2)\bigr)\Bigr)$$
$$=F'\Bigl(f\bigl(\rho_d(N)(s_1)\bigr),f\bigl(\rho_d(N)(s_2)\bigr)\Bigr)$$
$$=\rho_{d'}(N)\Bigl(\car_{B/A}(F')
\bigl(\car_{B/A}(f)(s_1),\car_{B/A}(f)(s_2)\bigr)\Bigr).$$
\end{proof}

Suppose that $A$ is of characteristic 0. Then
$\caa=A\otimes_{\Z}\Q$ and $\cab=B\otimes_{\Z}\Q$ are
$\Q$-algebras, and $\cab$ is an $\caa$-algebra of rank $n$.
We consider $\cab$ as a free $\caa$-module with the same
fixed basis $e_0,\ldots,e_{n-1}$.
Denote by $\lambda$ the logarithm of $F$.

\begin{log_wr} \label{log_wr}
$\car_{\cab/\caa}(\lambda)$ is the logarithm of $\car_{B/A}(F)$.
\end{log_wr}

\begin{proof} Denote $\cab_d=\cab[[x_1,...,x_d]]$ and
$\caa_d=\caa[[z_1,...,z_{nd}]]$. Let
$\rho^*_d:\Spf\caa_d\to\car_{\cab/\caa}(\Spf\cab_d)$ be defined
similarly to $\rho_d$. Since $\cab=\caa\otimes_AB$, the formal
group schemes $\car_{\cab/\caa}(\Spf \cab_d)$ and $\car_{B/A}(\Spf
B_d)_{\caa}$ coincide as well as the maps $\rho^*_d$ and
$(\rho_d)_{\caa}$. Hence,
$\car_{\cab/\caa}(F_{\cab})=\car_{B/A}(F)_{\caa}$. Besides, it is
clear that $\car_{\cab/\caa}((\F_a)^d_{\cab})=(\F_a)^{nd}_{\caa}$.
According to Proposition~\ref{hom_wr}, we get
$\car_{\cab/\caa}(\lambda)
\in\Hom_{\caa}(\car_{B/A}(F)_{\caa},(\F_a)^{nd}_{\caa})$. By
definition of $\car_{\cab/\caa}$, we have
$\sum_{j=1}^{n-1}\car_{\cab/\caa}(\lambda)_{jd+l}e_j\equiv
\sum_{i=1}^{n-1}z_{id+l}e_i\mod\deg2$. Hence,
$\car_{\cab/\caa}(\lambda)_{jd+l}\equiv z_{jd+l}\mod\deg2$. Thus
$\car_{\cab/\caa}(\lambda)$ is the logarithm of $\car_{B/A}(F)$.
\end{proof}

The next proposition shows a relation between two Weil restrictions
of the same formal group law defined with the aid of two distinct bases.

\begin{tran_wr} \label{tran_wr}
Let $e_0,\ldots, e_{n-1}$ and $ e'_0,\ldots, e'_{n-1}$
be two free $A$-bases of $B$ and $F$ be a formal group law over $B$.
If $\rho_d:\Spf A_d\to \car_{B/A}(\Spf B_d)$ and
$\rho'_d:\Spf A'_d\to \car_{B/A}(\Spf B_d)$ are
the corresponding maps, $\car_{B/A}(F)$ and $\car'_{B/A}(F)$ are
the corresponding formal group laws over $A$, then there exists
$f\in\Hom_A(\car'_{B/A}(F),\car_{B/A}(F))$ with linear coefficient
$I_d\otimes W$, where $W$ is the transition matrix from
$e'_0,\ldots,e'_{n-1}$ to $e_0,\ldots, e_{n-1}$.
\end{tran_wr}

\begin{proof}
Suppose that $W=\{w_{i,j}\}_{0\leq i,j\leq n-1}\in\GL_n(A)$,
i.e., $e'_j=\sum_{j=0}^{n-1}w_{i,j}e_i$.
Define the continuous homomorphism  $g:A_d\to A'_d$
as follows: $g(z_{id+l})=\sum_{j=0}^{n-1}w_{i,j}z'_{jd+l}$.
Let $N$ be a nilpotent $A$-algebra and $s\in\Hom_A(A'_d,N)$.
Then
$$
\rho_d(N)(\Spf g(s))(x_l)=\sum_{i=0}^{n-1}s(g(z_{id+l}))\otimes e_i
=\sum_{i=0}^{n-1}\sum_{j=0}^{n-1}w_{i,j}s(z'_{jd+l})\otimes e_i.
$$
On the other hand
$$
\rho'_d(N)(s)(x_l)=\sum_{j=0}^{n-1}s(z'_{jd+l})\otimes e'_j
=\sum_{j=0}^{n-1}\sum_{i=0}^{n-1}s(z'_{jd+l})\otimes w_{i,j}e_i.
$$
Thus $\rho_d\circ (\Spf g)=\rho'_d$ and therefore $\Spf g$ induces
a homomorphism
from $\car'_{B/A}(F)$ to $\car_{B/A}(F)$ whose linear
coefficient is equal to $I_d\otimes W$.
\end{proof}

\section{Galois action on the Weil restriction
of split tori}

Let $A$ be a ring, $B$ be an $A$-algebra which is a free
$A$-module of finite rank. Let $S$ be a smooth separated
scheme over $B$ of finite type. Define {\it Weil restriction}
functor $\car_{B/A}(S)$ from the category of $A$-algebras to
the category of sets by $\car_{B/A}(S)(R)=S(R\otimes_A B)$.

\begin{wr_sch}\cite[Section 7.6, Theorem 4, Proposition 5]{BLR}\label{wr_sch}
The functor $\car_{B/A}(S)$ is a separated smooth scheme over $A$.
\end{wr_sch}

If $G$ is a group scheme over $A$ such that its underlying scheme $S$
is smooth, separated and of finite type, then $\car_{B/A}(S)(R)$
is a group, which defines a group scheme denoted by $\car_{B/A}(G)$.

Let $K/L$ be a finite Galois extension of fields, and $T$ be a torus
over $L$ represented by the Hopf algebra $H$. It is easy to see
that if a field $K'$ splits $T$, then $KK'$ splits $\car_{K/L}(T_K)$,
and in particular, the latter scheme is a torus. Define a right
action of ${\rm Gal}(K/L)$ on $\car_{K/L}(T_K)$ as follows:
if $\sigma\in{\rm Gal}(K/L)$, $R$ is an $L$-algebra and
$s\in\car_{K/L}(T_K)(R)=
\Hom_K(H\otimes_LK,R\otimes_LK)$, then
$s\sigma=\hat\sigma^{-1}\circ s\circ \tilde\sigma$,
where $\tilde\sigma:H\otimes_L K\to H\otimes_L K$ and
$\hat\sigma:R\otimes_LK\to R\otimes_LK$ are induced by $\sigma$.
For every $L$-algebra $R$, the map
$$
\Omega(R):T(R)
\to\car_{K/L}(T_K)(R)
$$
defined by $\Omega(R)(g)(a\otimes v)=g(a)(1\otimes v)$,
$a\in H$, $v\in K$, is a homomorphism.
Hence, it gives a  morphism $\Omega:T\to\car_{K/L}(T_K)$.






Suppose that $L$ is a local or global field 
and denote its ring of integers by $\cao_L$. Let $S$ be a smooth
separated scheme over $L$. A smooth separated scheme $\cay$ over
$\cao_L$ is called a {\it N\'eron model} of $S$, if it is a model of
$S$, i.e. $\cay_L=S$, and satisfies the following universal property:
for any smooth scheme $\caz$ over $\cao_L$ and any $L$-morphism
$g:\caz_L\to \cay_L=S$ there exists a unique $\cao_L$-morphism
$h:\caz\to \cay$ such that $h_L=g$ (see \cite{BLR}, Section 1.2,
Definition 1). Evidently, if a scheme $S$ admits a
N\'eron model, then it is unique up to isomorphism.


\begin{mod_gr}\cite[Section 1.2, Proposition 6]{BLR}\label{mod_gr}
Let $G$ be a smooth separated group scheme over $L$ such that
its underlying scheme $S$ admits a N\'eron model $\cay$ over $\cao_L$.
Then $\cay$ admits a unique group scheme structure which induces
the original group scheme structure on $\cay_L=S$.
\end{mod_gr}



\begin{tor_mod}\cite[Section 10.1, Proposition 6]{BLR}\label{tor_mod}
Any torus over $L$ admits a N\'eron model over $\cao_L$.
\end{tor_mod}

By the above Proposition, the schemes $T$ and $\car_{K/L}(T_K)$
admit N\'eron models which we denote by $\cat$ and $\cau$,
respectively. Due to the universal property of N\'eron models,
the right action of $\Gal(K/L)$ on $\car_{K/L}(T_K)$ can be
extended to the right action on $\cau$, and $\Omega$ can be
extended to the morphism $\omega:\cat\to\cau.$

Let $S$ be a separated scheme over a ring $A$ provided
with an action of a finite group $\cag$. Define
the functor $S^\cag$ of fixed points by $S^\cag(R)=S(R)^\cag$.

\begin{fix_sch}\cite[Proposition 3.1]{Ed}\label{fix_sch}
The functor $S^\cag$ is represented by a closed subscheme of $S$.
\end{fix_sch}

\begin{tam_iso}\cite[Theorem 4.2]{Ed}\label{tam_iso}
If $K/L$ is tamely ramified, then
$\omega:\cat\to\cau$ is a closed immersion which induces
an isomorphism $\cat\to\cau^{\Gal(K/L)}$.
\end{tam_iso}

For a group scheme $G$, we denote by $G_0$ the connected
component of the unit in $G$.

The scheme $\cau_0$ is invariant with respect to the right
action of $\Gal(K/L)$ on $\cau$, so we obtain an action
on $\cau$.

\begin{cc_iso}\label{cc_iso}
The natural morphism
$({(\cau_0)}^{\Gal(K/L)})_0\to(\cau^{\Gal(K/L)})_0 $
is an isomorphism.
\end{cc_iso}

\begin {proof}
We describe how to construct the inverse morphism.
There is a natural morphism from $(\cau^{\Gal(K/L)})_0$
to $\cau_0$, and its image is a subscheme invariant
with respect to the action of ${\Gal(K/L)}$. That
gives a morphism from $(\cau^{\Gal(K/L)})_0$ to
${(\cau_0)}^{\Gal(K/L)}$, and the image of this morphism
is a connected subscheme. Thus we obtain a morphism from
$(\cau^{\Gal(K/L)})_0$ to $({(\cau_0)}^{\Gal(K/L)})_0$.
\end{proof}

Let $\cao_K$ denote the ring of integers of $K$.

\begin{wr_mfg}\cite[Lemma 3.1]{NX}\label{wr_mfg}
There exists a natural isomorphism from
$\car_{\cao_K/\cao_L}((\G_m)_{\cao_K}^d)$
to the connected component of the unit in
the N\'eron model for $\car_{K/L}((\G_m)_K^d)$.
\end{wr_mfg}

Suppose that $T$ is split over $K$, and its dimension is $d$.
Denote by $\cax$ the group of group-like elements in $H\otimes_LK$.
Then $\cax$ is a free $\Z$-module of rank $d$. We fix a free
$\Z$-basis $x_1,\ldots,x_d$ in $\cax$, and denote by $X$ the set
of variables $x_1,\ldots,x_d$. Then $\cax$ is identified
with $\Z^d$, $H\otimes_LK$ is identified with
the Hopf algebra $K[X,X^{-1}]$,
$T_K$ is identified with $(\mathbb{G}_m)_K^d$, $\cau$ is
identified with the N\'eron model for $\car_{K/L}((\G_m)_K^d)$,
and according to Proposition~\ref{wr_mfg}, $\cau_0$ is
identified with $\car_{\cao_K/\cao_L}((\G_m)_{\cao_K}^d)$.
Moreover we obtain a right action of $\Gal(K/L)$ on
$\car_{\cao_K/\cao_L}((\G_m)_{\cao_K}^d)$.
If $\sigma\in\Gal(K/L)$, $R$ is an $\cao_L$-algebra and
$$
s\in\car_{\cao_K/\cao_L}((\G_m)_{\cao_K}^d)(R)
=\Hom_{\cao_K}(\cao_K[X,X^{-1}],R\otimes_{\cao_L}\cao_K)
$$
then $s\sigma=\hat{\sigma'}^{-1}\circ s\circ \tilde\sigma'$,
where $\tilde\sigma'$ is the restriction of $\tilde\sigma$
to $\cao_K[X,X^{-1}]$, and $\hat\sigma'$ is induced by $\sigma$.

Let $T'$ be another torus over $L$ split over $K$ with
Hopf algebra $H'$ and $\eta\colon T\to T'$ be a morphism.
Denote by $\cat'$ and $\cau'$ the N\'eron models for $T'$
and $\car_{K/L}(T'_K)$. Then $\eta$ induces a morphism
$\cau_0\to\cau'_0$ which commutes with the actions of
$\Gal(K/L)$ on $\cau_0$ and $\cau'_0$. If a $\Z$-basis
$x'_1,\ldots,x'_{d'}$ in the group $\cax'$ of group-like
elements in $H'\otimes_LK$ is fixed, then $\eta$ induces
a morphism
$$
\tilde\eta\colon\car_{\cao_K/\cao_L}((\G_m)_{\cao_K}^d)
\to\car_{\cao_K/\cao_L}((\G_m)_{\cao_K}^{d'})
$$
which also commutes with the actions of $\Gal(K/L)$.

Suppose that $\eta$ corresponds to the homomorphism
$\cax'\to\cax$ given by the matrix
$C=\{c_{k,l}\}_{1\leq k\leq d;1\leq l\leq d'}\in M_{d,d'}(\Z)$
in the bases $x'_1,\ldots,x'_{d'}$ and $x_1,\ldots,x_d$.
If $R$ is an $\cao_L$-algebra and
$$
s\in\car_{\cao_K/\cao_L}((\G_m)_{\cao_K}^d)(R)
=\Hom_{\cao_K}(\cao_K[X,X^{-1}],R\otimes_{\cao_L}\cao_K),
$$
then $\tilde\eta(R)(s)(x'_l)=\sum_{k=1}^d c_{k,l}s(x_k)$.

\section{Main construction}

We keep the notations of the previous section.
Suppose that the characteristic of $L$ is equal to 0.
Denote the degree of $K/L$ by $n$ and fix a free $\cao_L$-basis
$e_0,\ldots,e_{n-1}$ of $\cao_K$ such that $ e_0=1$. Further, denote
$$
\cah= \cao_L[z_{j,l},z'_l]_{\scriptstyle 0\leq j\leq n-1
\atop\scriptstyle1\leq l\leq d}/\left(1-z'_l\can_{\cao_K/\cao_L}
\left(\sum_{0\le j\le n-1}z_{j,l} e_j\right)\right),
$$
where $\can_{\cao_K/\cao_L}$ is the norm map from $\cao_K$ to $\cao_L$.
For any $\cao_L$-algebra $R$, the map
$$
\nu(R):\Sp\cah(R)\to \car_{\cao_K/\cao_L}
(\Sp \cao_K[X,X^{-1}])(R)
$$
defined by
$\nu(R)(s)(x_l)=\sum_{m=0}^{n-1}s(z_{m,l})\otimes e_m$,
$1\leq l\leq d$, is a bijection. Therefore, it gives an isomorphism
$\nu:\Sp\cah\to \car_{\cao_K/\cao_L}((\G_m)_{\cao_K}^d)$
which allows to define the group scheme structure on $\Sp\cah$.

The augmentation ideal $\caj$ of the Hopf algebra $\cah$ is generated
by the elements $z_1,\ldots,z_{nd}$, where $z_{l}=z_{0,l}-1$,
$z_{jd+l}=z_{j,l}$, for $1\leq l\leq d$, $0<j\leq n-1$.
Further, $z_1+\caj^2,\ldots,z_{nd}+\caj^2$ form a free
$\cao_L$-basis of $\caj/\caj^2$. Thus the elements $z_1,\ldots,z_{nd}$
provide a coordinate system on $\wh{\Sp\cah}=\Spf\hat\cah_\caj$
and give rise to an $nd$-dimensional formal group law $\Phi $ with
logarithm $\Lambda$.

\begin{wr_mgs} \label{wr_mgs}
$\Phi =\car_{\cao_K/\cao_L}((\F_m)^d_{\cao_K})$.
\end{wr_mgs}

\begin{proof}
Let $N$ be a nilpotent $\cao_L$-algebra.
Denote $N'=N\otimes_{\cao_L}\cao_K$.
The set $\wh{\Sp\cah}(N)$ is
the subset of $\Sp\cah(\cao_L\oplus N)$ which consists of
the elements $s\in\Hom_{\cao_L}(\cah,\cao_L\oplus N)$ such that
$s(z_{0,l})-1,s(z_{j,l})\in N$,
$1\leq l\leq d$, $0<j\leq n-1$. On the other hand, the set
$\car_{\cao_K/\cao_L}((\F_m)^d_{\cao_K})(N)=(\F_m)^d_{\cao_K}(N')$
is the subset of $(\G_m)^d_{\cao_K}(\cao_K\oplus  N')$
which consists of the elements
$s'\in\Hom_{\cao_K}(\cao_K[X,X^{-1}],\cao_K\oplus N')$
such that $s'(x_l)-1\in N'$, $1\leq l\leq d$. Hence,
$\nu(\cao_L\oplus N)$ provides a bijection from $\wh{\Sp\cah}(N)$ to
$(\F_m)^d_{\cao_K}(N')$. Moreover, we have
$\hat\cah_\caj=\caa_d$ and $\nu(\cao_L\oplus N)$ restricted to
$\wh{\Sp\cah}(N)$ coincides with $\rho_d(N)$. The formal group laws
$\Phi $ and $\car_{\cao_K/\cao_L}((\F_m)^d_{\cao_K})$ come from
the group structure on $(\F_m)^d_{\cao_K}(N')$ with the aid
of the bijections $\nu(\cao_L\oplus N)$ and $\rho_d(N)$, respectively.
Hence, they coincide.
\end{proof}

\begin{wr_log} \label{wr_log}
For any $1\leq l\leq d$,
$$
\sum_{j=0}^{n-1}\Lambda_{jd+l}(z_1,\ldots,z_{nd})e_j=
\LL \left(\sum_{i=0}^{n-1}z_{id+l}e_i\right)
\in K[[z_1,\ldots,z_{nd}]].
$$
\end{wr_log}

\begin{proof}
It follows from Propositions~\ref{log_wr} and \ref{wr_mgs}.
\end{proof}

Let $\cah'$, $\nu'$, $\caj'$, $\Phi'$ be defined for
the torus $T'$ similar to $\cah$, $\nu$, $\caj$, $\Phi$.
The morphism $\eta\colon T\to T'$ induces a homomorphism
$\Phi\to\Phi'$ which corresponds to the morphism
$\nu'^{-1}\circ\tilde\eta\circ\nu\colon \Sp\cah\to\Sp\cah'$
in the bases $z_1,\ldots,z_{nd}$ and $z'_1,\ldots,z'_{nd'}$.

\begin{linco}\label{linco}
The linear coefficient of the homomorphism $\Phi\to\Phi'$ of
formal group laws induced by $\eta$ is $C^T\otimes I_n$.
\end{linco}

\begin{proof}
Let $R$ be a $\cao_L$-algebra and $s\in\Hom_{\cao_L}(\cah,R)$.
Then
$$
(\tilde\eta\circ\nu)(R)(s)(x'_l)=
\sum_{k=1}^dc_{k,l}\nu(R)(s)(x_k)=
\sum_{m=0}^{n-1}\sum_{k=1}^dc_{k,l}s(z_{m,k})\otimes e_m.
$$
On the other hand,
$$
(\tilde\eta\circ\nu)(R)(s)(x'_l)=
(\nu'\circ{\nu'}^{-1}\circ\tilde\eta\circ\nu)(R)(s)(x'_l)=
\sum_{m=0}^{n-1}
({\nu'}^{-1}\circ\tilde\eta\circ\nu)(R)(s)(z'_{m,l})\otimes e_m.
$$
Therefore
$({\nu'}^{-1}\circ\tilde\eta\circ\nu)(z'_{m,l})=\sum_{k=1}^dc_{k,l}z_{m,k}$.
Hence, the homomorphism $\caj'/{\caj'}^2\to\caj/\caj^2$ induced
by ${\nu'}^{-1}\circ\tilde\eta\circ\nu$ maps $z'_{md'+l}$ to
$\sum_{k=1}^dc_{k,l}z_{md+k}$, i.e., the matrix of this
homomorphism in the bases $z'_1,\ldots,z'_{nd'}$ and
$z_1,\ldots,z_{nd}$ is $C\otimes I_n$.
\end{proof}

The isomorphism $\nu$ allows to define a right action of
$\Gal(K/L)$ on $\Sp\cah$, and hence, an action on $\cah$ and on
$\caj/\caj^2$. In the basis $z_1,\ldots,z_{nd}$, the action on
$\caj/\caj^2$ gives the representation
$\theta:\Gal(K/L)\to\GL_{nd}(\cao_L)$.

The group $\cax$ is invariant with respect to
the action of the  group $\Gal(K/L)$ on $H\otimes_LK$
which is defined by the intrinsic action on $K$.
Thus in the basis $x_1,\ldots,x_d$, it provides
the representation $\chi:\Gal(K/L)\to\GL_d(\Z)$.

There is a unique action of $\Gal(K/L)$ on
$\tilde\cao_K:=\Hom_{\cao_L}(\cao_K,\cao_L)$ such that
$\sigma \tilde a(\sigma a)= \tilde a(a)$ for any $a\in \cao_K$,
$\tilde a\in\tilde\cao_K$, $\sigma\in\Gal(K/L)$.
The elements $\tilde e_0,\ldots,\tilde e_{n-1}$
defined by $\tilde e_j( e_i)=\delta_i^j$ form a free
$\cao_L$-basis of $ \tilde\cao_K$ which gives
the representation $\psi:\Gal(K/L)\to\GL_n(\cao_L)$.

\begin{q_mat} \label{q_mat}
$\theta(\sigma)=\chi(\sigma)\otimes\psi(\sigma)$
for any $\sigma\in\Gal(K/L)$.
\end{q_mat}

\begin{proof}
Suppose that $\chi(\sigma)=\{a_{k,l}\}_{1\leq k,l\leq d}$,
$\psi(\sigma)=\{b_{i,j}\}_{0\leq i,j\leq n-1}$.
Let $R$ be an $\cao_L$-algebra and $s\in\Hom_{\cao_L}(\cah,R)$.
Then
$$\nu(R)(s)\in
\Hom_{\cao_K}(\cao_K[X,X^{-1}],R\otimes_{\cao_L}\cao_K).
$$
Besides, we have the natural mapping
$ \tilde\cao_K\times R\otimes_{\cao_L}\cao_K\to R$. Then
$$
\tilde e_j(-1+((\nu(R)(s))\sigma)(x_l))=
\sigma \tilde e_j(-1+(\nu(R)(s))(\tilde\sigma x_l))
$$
$$
=\sum_{i=0}^{n-1}b_{i,j} \tilde e_i\left(-1+(\nu(R)(s))
\left(\prod_{k=1}^dx_k^{a_{k,l}}\right)\right)
$$
$$
=\sum_{i=0}^{n-1}b_{i,j} \tilde e_i
\left(-1+\prod_{k=1}^d\left(1+\sum_{m=0}^{n-1}
s(z_{md+k})\otimes e_m\right)^{a_{k,l}}\right)
$$
$$
\equiv\sum_{i=0}^{n-1}b_{i,j} \tilde e_i
\left(\sum_{k=1}^da_{k,l}\sum_{m=0}^{n-1}s( z_{md+k})\otimes e_m\right)
$$
$$
=\sum_{i=0}^{n-1}\sum_{k=1}^db_{i,j}a_{k,l}s( z_{id+k}) \mod s(\caj^2).
$$
On the other hand,
$$ \tilde e_j(-1+(\nu(R)(s\sigma))(x_l))=
\tilde e_j\left(\sum_{m=0}^{n-1}(s\sigma)( z_{md+l})\otimes e_m\right)=
(s\sigma)( z_{jd+l}).
$$
Thus
$\sigma( z_{jd+l})=\sum_{i=0}^{n-1}\sum_{k=1}^db_{i,j}a_{k,l} z_{id+k}$.
\end{proof}

\begin{q_mat+} \label{q_mat+}
The $\cao_L$-linear map
$\Theta: \tilde\cao_K\otimes_{\cao_L}(\cax\otimes_{\Z}\cao_L)\to\caj/\caj^2$
defined by $\Theta( \tilde e_j\otimes x_l)=z_{jd+l}$,
is $\Gal(K/L)$-equivariant.
\end{q_mat+}





The right action of $\Gal(K/L)$ on $\Sp\cah$ induces a right
action on $\Phi$ which is in fact an action of the opposite group
$\Gal(K/L)^\circ$ on $\Phi$. Clearly, the linear coefficient of the
endomorphism $\sigma\in\Gal(K/L)^\circ$ of $\Phi$ is $\theta(\sigma)^T$.

\begin{repr}\label{repr}
If $K/L$ is tamely ramified, then there exists a universal fixed
pair $(F,f)$ for $(\Phi,\Gal(K/L)^\circ)$ such that $F$ represents
$\hat\cat$.
\end{repr}

\begin{proof}
According to Proposition~\ref{tam_iso}, the restriction
of $\omega$ to $\cat_0$ is an isomorphism from
$\cat_0$ to $(\cau^{\Gal(K/L)})_0$. Proposition~\ref{cc_iso}
implies that $\cat_0$ is isomorphic to $((\cau_0)^{\Gal(K/L)})_0$.
Since $\cau_0$ is identified with
$\car_{\cao_K/\cao_L}((\G_m)_{\cao_K}^d)$, and $\nu$
is an isomorphism, we obtain that $\cat_0$ is
isomorphic to $((\Sp\cah)^{\Gal(K/L)})_0$. Therefore
$\hat\cat$ is isomorphic to $\wh{(\Sp\cah)^{\Gal(K/L)}}$.
Moreover, $(\Sp\cah)^{\Gal(K/L)}$ is represented by the algebra
$\cah/\cai$, where $\cai$ is the ideal generated by elements
$a-\sigma a$, $a\in\cah$, $\sigma\in\Gal (K/L)$.
Consider a formal group law $F$ which represents
$\wh{(\Sp\cah)^{\Gal(K/L)}}$. The morphism
$\iota:(\Sp\cah)^{\Gal(K/L)}\to(\Sp\cah)$ induces a homomorphism
$f=\hat\iota: F\to \Phi$. Evidently, $(F,f)$ is a fixed pair for
$(\Phi,\Gal(K/L)^\circ)$. For any nilpotent $\cao_L$-algebra $N$,
$\wh{(\Sp\cah)^{\Gal(K/L)}}(N)=\Hom^*_A(\cah/\cai,A\oplus N)$ consists
of elements $s\in\Hom^*_A(\cah,A\oplus N)$ such that $sa-\sigma a=0$
for any $a\in\cah$, $\sigma\in\Gal (K/L)$, and hence, it coincides
with $\wh{\Sp\cah}(N)^{\Gal(K/L)}$. Therefore the map
$f(N):\wh{(\Sp\cah)^{\Gal(K/L)}}(N)\to \wh{\Sp\cah}(N)^{\Gal(K/L)}$
is bijective, and by Proposition~\ref{ufix_suf},
$(F,f)$ is a universal fixed pair for $(\Phi,\Gal(K/L)^\circ)$.
\end{proof}

\section{Tori split over a tamely ramified abelian extension of $\Q_p$}

We keep the notations of Section 6. Suppose that $L=\Q_p$,
and $K/\Q_p$ is a tamely ramified abelian extension. In this
case $K=K_1K_2$, where $K_1/\Q_p$ is an unramified extension of
degree $n_1$, and $K_2/\Q_p$ is a totally ramified extension of
degree $n_2$, $n=n_1n_2$. Then $K_1=\Q_p(\xi)$, where
$\xi$ is a primitive $(p^{n_1}-1)$-th root of unity; $K_2=\Q_p(\pi)$,
where $\pi^{n_2}=p\epsilon$, $\epsilon\in\Z_p^*$ and $n_2\mid p-1$.
The group $\Gal(K/K_2)$ is isomorphic to $\Gal(K_1/\Q_p)\cong\Z/n_1\Z$.
Take $\sigma_1\in\Gal(K/K_2)$ such that $\sigma_1|_{K_1}$
is the Frobenius automorphism. Then $\sigma_1$ generates $\Gal(K/K_2)$.
The inertia subgroup $\Gal(K/K_1)$ of $\Gal(K/\Q_p)$ is
isomorphic to $\Gal(K_2/\Q_p)\cong\Z/n_2\Z$.
Let $\zeta$ be a primitive $n_2$-th root of unity.
Take $\sigma_2\in\Gal(K/K_1)$ such that $\sigma_2(\pi)=\zeta\pi$.
Then $\sigma_2$ generates $\Gal(K/K_1)$.
Consider the following $\Z_p$-basis of $\cao_K$:
$e_{in_2+j}=\pi^j\xi^{p^i}$, $0\leq i\leq n_1-1$,
$0\leq j\leq n_2-1$.

If $\kappa\in K[[y_1,\ldots,y_m]]$ then there exist unique
$\kappa^{(j)}\in K_1[[y_1,\ldots,y_m]]$, $0\leq j\leq n_2-1$,
such that $\kappa=\sum_{j=0}^{n_2-1}\kappa^{(j)}\pi^j$.

\begin{sm_ram}\cite[Proposition 1.7, Proposition 1.8]{De}\label{sm_ram}
Let $\lambda=\sum_{i=1}^\infty c_iy^i\in K[[y]]$ be such that
$p^kc_{rp^k}\in \cao_K$ for any non-negative integers $r$, $k$.
If $\phi\in \cao_K[[y_1,\ldots,y_m]]$, then

\noindent{\rm (i)}
$(\lambda\circ\phi)^{(0)}\equiv\lambda^{(0)}\circ\phi^{(0)}\mod p$;

\noindent{\rm (ii)}
$(\lambda\circ\phi)^{(j)}\equiv\lambda^{(j)}\circ\phi^{(0)}\mod \cao_K$
for $0<j\leq n_2-1$.
\end{sm_ram}

\begin{tip3}\label{tip3}
$\Lambda$ is of type
$v=pI_{nd}-I_d\otimes J_{n_2}\otimes P_{n_1}\bt$.
\end{tip3}

\begin{proof}
For any $1\le l\le d$, consider
$$
\phi_l=\sum_{i=0}^{n_1-1}\sum_{j=0}^{n_2-1}
z_{(in_2+j)d+l}e_{in_2+j}\in\cao_K[[z_l,z_{d+l},\ldots,z_{(n-1)d+l}]].
$$
By Corollary from Proposition~\ref{wr_mgs}, we have
$$
\sum_{j=0}^{n_2-1}\sum_{i=0}^{n_1-1}\Lambda_{(in_2+j)d+l}e_{in_2+j}
=\LL \circ\phi_l
$$
which implies
$\sum_{i=0}^{n_1-1}\Lambda_{(in_2+j)d+l}\xi^{p^i}=
(\LL \circ\phi_l)^{(j)}$ for any $0\le j\le n_2-1$.

If $j\ne 0$, then by Lemma~\ref{sm_ram}
$$
\sum_{i=0}^{n_1-1}\Lambda_{(in_2+j)d+l}\xi^{p^i}\equiv
\LL ^{(j)}\circ\phi_l^{(0)}=0\mod \cao_K,
$$
whence  $p\Lambda_{(in_2+j)d+l}\equiv0\mod p$.

If $j=0$, then Lemma~\ref{sm_ram} implies
$\Lambda_{in_2d+l}\xi^{p^i}\equiv
\LL \circ\phi_l^{(0)}\mod p$.

We consider the action of $\bt$ on $K_1[[z_1,\ldots,z_{nd}]]$
introduced in Section 2. Then applying Lemma~\ref{ass} we get
$$
(p-\bt) \sum_{i=0}^{n_1-1}\Lambda_{in_2d+l}\xi^{p^i}\equiv
(p-\bt)(\LL \circ\phi_l^{(0)})\equiv
\left((p-\bt)\LL \right)\circ\phi_l^{(0)}\equiv 0\mod p.
$$
Therefore
$$
p\sum_{i=0}^{n_1-1}\Lambda_{in_2d+l}\xi^{p^i}\equiv
\sum_{i=0}^{n_1-1}(\bt\Lambda_{in_2d+l})\xi^{p^{i+1}}\mod p.
$$
Hence $p\Lambda_{in_2d+l}-\bt\Lambda_{(i-1)n_2d+l}\equiv0\mod p$ for
$1\le i\le n_1-1$ and
$p\Lambda_l-\bt\Lambda_{(n_1-1)n_2d+l}\equiv0\mod p$.

Thus $p\Lambda\equiv W\bt\Lambda\mod p$ and
$W=\{w_{\alpha,\beta}\}_{1\le\alpha,\beta\le nd}$ where
$w_{in_2+l,(i-1)n_2d+l}=1$ for $1\le i\le n_1-1, 1\le l\le d$ and
$w_{l,(n_1-1)n_2d+l}=1$ for $1\le l\le d$,
the other entries being equal to 0.
Therefore $W=I_d\otimes J_{n_2}\otimes P_{n_1}$ as required.
\end{proof}

For $0\leq i\leq n_1-2$, $0\leq j\leq n_2-1$, we have
$\sigma_1(e_{in_2+j})=e_{(i+1)n_2+j}$,
$\sigma_1(e_{(n_1-1)n_2+j})=e_j$, i.e., in the basis
$e_0,\ldots,e_{n-1}$ the automorphism $\sigma_1$ is given
by the matrix $I_{n_2}\otimes P_{n_1}$. Therefore,
$\psi(\sigma_1)=((I_{n_2}\otimes P_{n_1})^{-1})^T=I_{n_2}\otimes P_{n_1}$.
Denote $U_1=\chi(\sigma_1)^T$. Then by Theorem~\ref{q_mat},
$\theta(\sigma_1)^T=U_1\otimes I_{n_2}\otimes P_{n_1}^T$.

\begin{v_unr}\label{v_unr}
If
$$
\hat C=
\begin{pmatrix}
U_1^{-1}\otimes I_{n_2}\\
U_1^{-2}\otimes I_{n_2}\\
\vdots\\
U_1^{-n_1+1}\otimes I_{n_2}
\end{pmatrix},\quad
\quad Q_1=
\begin{pmatrix}
I_{n_2d} & 0\\
\hat C & I_{nd-n_2d}
\end{pmatrix},
$$
and $D=U_1\otimes I_{n_2}\otimes P_{n_1}^T$,
then $Q_1^{-1}DQ_1-I_{nd}=\begin{pmatrix}
0 & \hat D\\
0 & \tilde D
\end{pmatrix},
$
where
$\hat D\in\M_{n_2d,nd-n_2d}(\Z)$,
$\tilde D\in\M_{nd-n_2d}(\Z)$,
and $\hat C\hat D+\tilde D$ is invertible.
\end{v_unr}
\begin{proof}
First, notice that
$$
U_1\otimes I_{n_2}\otimes P_{n_1}^T=
\begin{pmatrix}
0      & U_1\otimes I_{n_2} & 0 & \cdots & 0 \\
0      & 0 & U_1\otimes I_{n_2} &        & 0\\
\vdots &   &   & \ddots & \vdots\\
0      & 0 & 0 &        & U_1\otimes I_{n_2}\\
U_1\otimes I_{n_2}      & 0 & 0 & \cdots & 0
\end{pmatrix}.
$$
Then one can directly compute that $Q_1^{-1}DQ_1-I_{nd}$
is of the required form with
$\hat D=
\begin{pmatrix}
U_1\otimes I_{n_2} & 0 & 0 & \cdots & 0
\end{pmatrix}
\in\M_{n_2d,nd-n_2d}(\Z)$, and
$$
\tilde  D=
\begin{pmatrix}
-2I_{n_2d}                  &U_1\otimes I_{n_2}& 0              &\cdots& 0& 0\\
-U_1^{-1}\otimes I_{n_2}    &-I_{n_2d}       &U_1\otimes I_{n_2}&    & 0& 0\\
-U_1^{-2}\otimes I_{n_2}    & 0              &-I_{n_2d} &        & 0    & 0\\
 \vdots                     &                &     & \ddots &      & \vdots \\
-U_1^{-n_1+3}\otimes I_{n_2}& 0         &0& & -I_{n_2d} & U_1\otimes I_{n_2}\\
-U_1^{-n_1+2}\otimes I_{n_2}& 0              & 0    & \cdots & 0    & -I_{n_2d}
\end{pmatrix} \in\M_{nd-n_2d}(\Z).
$$
The last assertion is now obvious.
\end{proof}

\begin{edlem}\cite[Lemma 3.3]{Ed}\label{edlem}
Let $\cag=\Z/n\Z$, $(n,p)=1$, and $\bar M$ be a finitely generated
$\F_p[\cag]$-module. Then there exists a unique up to isomorphism
$\Z_p[\cag]$-module $M$ which is free as a $\Z_p$-module and such that
$M\otimes_{\Z_p}\F_p\cong\bar M$.
\end{edlem}

\begin{eigen} \label{eigen}
Let $\cag=\Z/{n_2}\Z$, $n_2\mid p-1$, $\sigma$ be a generator
of $\cag$ and $M$ be a finitely generated $\Z_p[\cag]$-module
which is free as a $\Z_p$-module. Then there exists a free
$\Z_p$-basis of $M$ consisting of eigenvectors of $\sigma$.
The multiplicity of the eigenvalue 1 is equal to $\rk_{\Z_p}M^\cag$.
\end{eigen}

\begin{proof}
Since $(n_2,p)=1$, the algebra $\F_p[\cag]$ is semisimple by Mashke's
Theorem, and the $\F_p[\cag]$-module $M\otimes_{\Z_p}\F_p$ is
isomorphic to the direct sum of irreducible $\F_p[\cag]$-modules.
Let $\bar M_0$ be an irreducible $\F_p[\cag]$-module, and
$x\in\bar M_0$, $x\neq0$. Then
$0=(\sigma^{n_2}-1)x=\prod_{i=1}^{n_2}(\sigma-\zeta_i)x$,
where $\zeta_1,\ldots,\zeta_{n_2}$ are the $n_2$-th roots of unity.
Therefore, there exist $y\in M$, $y\ne 0$, and an index
$1\leq i\leq n_2$ such that $(\sigma-\zeta_i)y=0$. Obviously $\F_py$
is an $\F_p[\cag]$-submodule of $\bar M_0$. Since $\bar M_0$ is
irreducible and $y\ne 0$, we get $\F_py=M$, i.e., every irreducible
$\F_p[\cag]$-module is one-dimensional.
Thus there exists a free $\F_p$-basis $m_1,\ldots,m_d$
of $M\otimes_{\Z_p}\F_p$ consisting of eigenvectors of $\sigma$.
Let $\bar\alpha_i\in\F_p$ be the eigenvalue of the vector $m_i$,
$1\leq i\leq d$. Define a $\Z_p[\cag]$-module structure on $M'=\Z_p^d$
by $\sigma(a_1,\ldots,a_d)=(\alpha_1a_1,\ldots,\alpha_da_d)$,
where $\alpha_i\in\Z_p$ is the multiplicative representative of
$\bar\alpha_i$, $1\leq i\leq d$. Then
$M'\otimes_{\Z_p}\F_p\cong M\otimes_{\Z_p}\F_p$. Therefore,
Lemma~\ref{edlem} implies that $M\cong M'$.
\end{proof}

For $0\leq i\leq n_1-1$, $0\leq j\leq n_2-1$, we have
$\sigma_2(e_{in_2+j})=\zeta^je_{in_2+j}$ and
$\sigma_2(\tilde e_{in_2+j})=\zeta^{-j}\tilde e_{in_2+j}$, i.e.,
$\psi(\sigma_2)=V \otimes I_{n_1}$, where $V $ is
the diagonal matrix with $j$-th diagonal entry
equal to $\zeta^{-j+1}$.

Denote $d_0=\rk_{\Z_p}(\cax\otimes_{\Z}\Z_p)^{<\sigma_2>}$. Then
by Proposition~\ref{eigen}, there exists a free $\Z_p$-basis
$y_1,\ldots,y_d$ of $\cax\otimes_{\Z}\Z_p$ consisting of
eigenvectors of $\sigma_2$ such that the eigenvalues of
$y_1,\ldots,y_{d_0}$ are equal to $1$, and the eigenvalues of
$y_{d_0+1},\ldots,y_d$ are not equal to $1$. Let $W\in\GL_d(\Z_p)$
be the transition matrix from $y_1,\ldots,y_d$ to
$x_1,\ldots,x_d$. Then $W^{-1}\chi(\sigma_2)W$ is a diagonal
matrix with $l$-th diagonal entry equal to $\zeta^{i_l}$, where
$i_l=0$ for $1\leq l\leq d_0$, and $1\leq i_l\leq n_2-1$ for
$d_0+1\leq l\leq d$.

Denote $U_2=\chi(\sigma_2)^T$. Then by Theorem~\ref{q_mat},
$\theta(\sigma_2)^T=U_2\otimes V \otimes I_{n_1}$.

Define a permutation $\tau$ of the set $\{1,\ldots,n_2d\}$
in the following way: for  $0\le j\le n_2-1, 1\le l\le d$
put $\tau(jd+l)=j_1d+l$, where $j_1$ is the reminder of
the division of $j+i_l$ by $n_2$.
Then $\tau(i)=i$ for $1\leq i\leq d_0$ and $\tau(i)>d$ for
$d_0+1\le i\le d$.

\begin{v_ram}\label{v_ram}
If $P_\tau\in\M_{n_2d}(\Z)$ is the matrix of $\tau$, i.e.,
$P_\tau=\{\delta_\alpha^{\tau(\beta)}\}_{1\le\alpha,\beta\le n_2d}$,
$Q_2=((W^T)\otimes I_{n_2})^{-1}P_\tau$ and $D=U_2\otimes V$, then
$Q_2^{-1}DQ_2-I_{n_2d}=\begin{pmatrix}
0 & 0\\
0 & \tilde D
\end{pmatrix}$, where
$\tilde D\in\M_{n_2d-d}(\Z_p)$ is invertible.
\end{v_ram}

\begin{proof}
As we showed above,
$(W^T\otimes I_{n_2})D(W^T\otimes I_{n_2})^{-1}
=(W^{-1}\chi(\sigma_2)W)^T\otimes V$ is a diagonal
matrix with the $(jd+l)$-th diagonal entry equal to
$\zeta^{i_l-j}$ for $0\le j\le n_2-1, 1\le l\le d$.
According to the definition of $\tau$, we obtain that
$P_\tau^{-1}(W^T\otimes I_{n_2})
D(W^T\otimes I_{n_2})^{-1} P_\tau$
is a diagonal matrix with first $d$ diagonal entries equal
to $1$ and the rest of the diagonal entries equal to $\zeta^i$
for $1\le i\le n_2-1$. Hence, $Q_2^{-1}DQ_2-I_{n_2d}$ is of
the required form. Finally, since $\zeta^i-1$ is invertible
for $1\le i\le n_2-1$, we get that
$\tilde D\in\M_{n_2d-d}(\Z_p)$ is invertible.
\end{proof}

The character group of the maximal subtorus $T_s$
(resp. $T_a$) of $T$ which is split (resp. anisotropic)
over $K_1$ is canonically isomorphic to $\cax/\Ker\rho_s$
(resp. $\cax/\Ker\rho_a$),  where $\rho_s\colon\cax\to\cax$
and $\rho_a\colon\cax\to\cax$ are defined by
$$
\rho_s(x)=\sum_{\sigma\in\Gal(K/K_1)}\sigma x=
x+\sigma_2x+\cdots+\sigma_2^{n_2-1}x,\quad
\rho_a(x)=\sigma_2(x)-x.$$
Denote $d_s=\dim T_s$, $d_a=\dim T_a$. Then $d_s+d_a=d$
(see \cite[Section 7.4]{Wa}). Let $\tilde x_1,\ldots,\tilde x_{d_s}$
and $\hat x_1,\ldots,\hat x_{d_a}$
be free $\Z$-bases of $\cax/\Ker\rho_s$ and
$\cax/\Ker\rho_a$, respectively, and let
$\tilde\chi:\Gal(K_1/\Q_p)\to\GL_{d_s}(\Z)$
be the representation corresponding
to $\tilde x_1,\ldots,\tilde x_{d_s}$.
Denote $\tilde U_1=\tilde\chi(\sigma_1|_{K_1})^T$.

Let $T'_s$, $T'_a$, $\rho'_s$, $\rho'_a$, $d'_s $, $d'_a$
be defined for the torus $T'$ similar to
$T_s$, $T_a$, $\rho_s$, $\rho_a$, $d_s$, $d_a$.
Let $\tilde x'_1,\ldots,\tilde x'_{d'_s }$
and $\hat x'_1,\ldots,\hat x'_{d'_a}$
be free $\Z$-bases of $\cax'/\Ker\rho'_s$
and $\cax'/\Ker\rho'_a$, respectively.
The morphism $\eta\colon T\to T'$ induces
the morphisms $\eta_s\colon T_s\to T'_s$ and
$\eta_a\colon T_a\to T'_a$ which correspond to
homomorphisms $\cax'/\Ker\rho'_s\to\cax/\Ker\rho_s$
and $\cax'/\Ker\rho'_a\to\cax/\Ker\rho_a$.
Denote the matrices of these homomorphisms in
the chosen bases by
$C_s\in\M_{d_s,d'_s }(\Z)$ and
$C_a\in\M_{d_a,d'_a }(\Z)$, respectively.

\begin{tpt3}\label{tpt3}
Let $T$, $T'$ be tori over $\Q_p$ which are split
over an abelian tamely ramified extension $K$ of $\Q_p$,
$\cat$, $\cat'$ be their N\'eron models, and
$\eta\colon T\to T'$ be a morphism. Then

{\rm I.} A formal group law with logarithm of type
$$
pI_d-
\begin{pmatrix}
\tilde U_1 & 0\\
0 & 0\end{pmatrix}\bt
$$
represents $\hat\cat$.

{\rm II.}
The linear coefficient of the homomorphism of formal group
laws corresponding to $\hat\cat$ and $\hat\cat'$
in the sense of part {\rm I} induced by $\eta$ is equal to
$$
\begin{pmatrix}
C_s^T & 0\\
0 & C_a^T
\end{pmatrix}.
$$
\end{tpt3}

\begin{proof}
I. According to Theorem~\ref{repr}, a formal group law which appears
in a universal fixed pair for $(\Phi,\Gal(K/\Q_p))$ represents
$\hat\cat$. By Theorem~\ref{ufix_ex} we can take a universal
fixed pair $(F_1, f_1)$ for $(\Phi, \langle\sigma_1\rangle)$
such that the linear coefficient of $f_1$ is $Q_1I_{nd,n_2d}$,
and the logarithm of $F_1$ is of type $u_1$, where $u_1$ is
the upper-left $n_2d\times n_2d$-submatrix of $Q_1^{-1}vQ_1$
for $v$ and $Q_1$ being from Proposition~\ref{tip3} and
Lemma~\ref{v_unr}, respectively, i.e.,
$$
v=pI_{nd}-
\begin{pmatrix}
0                          & 0              &\cdots& 0& I_d\otimes J_{n_2}\\
I_d\otimes J_{n_2}         & 0               &     & 0&     0\\
0                          &I_d\otimes J_{n_2}&    &  0  & 0  \\
 \vdots                     &                &\ddots     & &\vdots \\
0& 0              & \cdots       & I_d\otimes J_{n_2}&0
\end{pmatrix}\bt,
$$
$$
Q_1=
\begin{pmatrix}
I_{n_2d}                    & 0              &\cdots& 0& 0\\
U_1^{-1}\otimes I_{n_2}     & I_{n_2d}       &      & 0& 0\\
U_1^{-2}\otimes I_{n_2}     & 0              &      & 0&  0 \\
 \vdots                     &                &\ddots&  & \vdots \\
U_1^{-n_1+1}\otimes I_{n_2}& 0              &\cdots& 0& I_{n_2d}
\end{pmatrix}.
$$
An easy calculation shows that $u_1=pI_{n_2d}-U_1\otimes
J_{n_2}\bt$. The linear coefficient of $f_1$ is equal to
$Q_1I_{nd,n_2d}$. The action of
$\langle\sigma_2\rangle$ on $\Phi$ induces the action on
$F_1$ by condition $f_1\circ\sigma'_2=\sigma_2\circ f_1$ for
$\sigma'_2\in\Aut_{\Z_p}(F_1)$. If $Z\in\M_{n_2d}(\Z_p)$ is its
linear coefficient, we have $I_{nd,n_2d}Z=Q_1^{-1}(U_2\otimes V
\otimes I_{n_1})Q_1I_{nd,n_2d}$, which implies that $Z$ is the
upper-left $n_2d\times n_2d$-submatrix of $Q_1^{-1}(U_2\otimes V
\otimes I_{n_1})Q_1$. An easy calculation gives $Z=U_2\otimes V $.

By Theorem~\ref{ufix_ex} we can take a universal fixed pair
$(F_2, f_2)$ for $(F_1,\langle\sigma'_2\rangle)$ such that
the linear coefficient of $f_2$ is $Q_2I_{n_2d,d}$, and
the logarithm of $F_2$ is of type $u_2$, where $u_2$ is
the upper-left $d\times d$-submatrix of
$$
Q_2^{-1}u_1Q_2=pI_{n_2d}-Q_2^{-1}(U_1\otimes J_{n_2})Q_2\bt
=pI_{n_2d}-P_\tau^{-1}(W^{-1}\chi(\sigma_1)W\otimes J_{n_2})^TP_\tau\bt
$$
for $Q_2$ being from Lemma~\ref{v_ram}. Properties of the
permutation $\tau$ imply that the upper-left $d\times d$-submatrix
of $P_\tau^{-1}(W^{-1}\chi(\sigma_1)W\otimes J_{n_2})^TP_\tau$
is equal to
$\begin{pmatrix}
\hat U_1 & 0\\
0 & 0
\end{pmatrix}
\in\M_d(\Z_p)$,
where $\hat U_1$ is the upper-left $d_0\times d_0$-submatrix
of $(W^{-1}\chi(\sigma_1)W)^T$.

If we extend $\rho_s$ and $\rho_a$ on $\cax\otimes_{\Z}\Z_p$,
then we get $\rho_s(y_l)=n_2y_l$ and $\rho_a(y_l)=0$
for $1\leq l\leq d_0$, $\rho_s(y_l)=0$ and
$\rho_a(y_l)=(\zeta^{i_l}-1)y_l$ for $d_0+1\leq l\leq d$,
and consequently,
$\Ker\rho_s=\langle y_{d_0+1},\ldots,y_d\rangle$,
$\Ker\rho_a=\langle y_1,\ldots,y_{d_0}\rangle$.
Moreover
$y_1+\Ker\rho_s,\ldots,y_{d_0}+\Ker\rho_s$ and
$y_{d_0+1}+\Ker\rho_a,\ldots,y_d+\Ker\rho_a$
are free $\Z_p$-bases of
$(\cax\otimes_{\Z}\Z_p)/\Ker\rho_s$ and
$(\cax\otimes_{\Z}\Z_p)/\Ker\rho_a$, respectively. Clearly,
$\tilde x_1,\ldots,\tilde x_{d_s}$ is a free $\Z_p$-basis of
$(\cax/\Ker\rho_s)\otimes_{\Z}\Z_p\cong(\cax\otimes_{\Z}\Z_p)/\Ker\rho_s$,
and $\hat x_1,\ldots,\hat x_{d_a}$ is a free $\Z_p$-basis of
$(\cax/\Ker\rho_a)\otimes_{\Z}\Z_p\cong(\cax\otimes_{\Z}\Z_p)/\Ker\rho_a$.
It yields in particular $d_0=d_s$.
Denote by $W_s\in\GL_{d_s}(\Z_p)$ the transition
matrix from $y_1+\Ker\rho_s,\ldots,y_{d_s}+\Ker\rho_s$
to $\tilde x_1,\ldots,\tilde x_{d_s}$ and by
$W_a\in\GL_{d_a}(\Z_p)$ the transition matrix
from $y_{d_s+1}+\Ker\rho_a,\ldots,y_d+\Ker\rho_a$
to $\hat x_1,\ldots,\hat x_{d_a}$.

The matrices $W^{-1}\chi(\sigma_1)W$ and $W^{-1}\chi(\sigma_2)W$
commute, therefore the former is a block diagonal matrix with
two blocks of size $d_s\times d_s$ and
$d_a\times d_a$.
One can easily see that its upper-left block coincides with
$W_s^{-1}\tilde\chi(\sigma_1)W_s$. It implies
$\hat U_1=W_s^T\tilde U_1(W_s^T)^{-1}$.
Thus
$$
u_2=pI_{n_2d}-\begin{pmatrix}
W_s^T\tilde U_1(W_s^T)^{-1} & 0\\
0 & 0
\end{pmatrix}\bt.
$$
Now $(F_2, f_1\circ f_2)$ is a universal
fixed pair for $(\Phi,\Gal(K/\Q_p) )$ by Proposition~\ref{ufix_prod}.

According to Proposition~\ref{hon} (i),(iii), there exists
a formal group law $F_3$ whose logarithm is of type
$$
\begin{pmatrix}
W_s^T & 0\\
0 & W_a^T
\end{pmatrix}^{-1}u_2
\begin{pmatrix}
W_s^T & 0\\
0 & W_a^T
\end{pmatrix}
=pI_d-
\begin{pmatrix}
\tilde U_1 & 0\\
0 & 0
\end{pmatrix}\bt,
$$
and
$$
f_3(X)=
\begin{pmatrix}
W_s^T & 0\\
0 & W_a^T
\end{pmatrix}X
$$
is an isomorphism from $F_3$ to $F_2$.
Thus, $(F_3, f_1\circ f_2\circ f_3)$ is a universal fixed
pair for $(\Phi, \Gal(K/\Q_p) )$ with $F_3$ as required.

II. The homomorphism $\Phi\to\Phi'$ induced by $\eta$ commutes
with the actions of $\Gal(K/L)$ on $\Phi$ and $\Phi'$.
According to Proposition~\ref{linco}, its linear coefficient
is equal to $C^T\otimes I_n$. Let $(F_1,f_1)$
and $(F'_1,f'_1)$ be as above.
Then by Proposition~\ref{mor_ufix}, we get a homomorphism
$F_1\to F'_1$ whose linear coefficient $Z_1$ satisfies
$Q'_1I_{nd',n_2d'}Z_1=(C^T\otimes I_n)Q_1I_{nd,n_2d}$. It implies
that $Z_1$ is the upper-left $n_2d'\times n_2d$-submatrix of
${Q'_1}^{-1}(C^T\otimes I_n)Q_1$. An easy calculation shows
that $Z_1=C^T\otimes I_{n_2}$. The homomorphism $F_1\to F'_1$
commutes with the actions of $\langle\sigma_2\rangle$ on $F_1$
and $F'_1$. Then Proposition~\ref{mor_ufix} yields a homomorphism
$F_2\to F'_2$ whose linear coefficient $Z_2$ satisfies
$Q'_2I_{n_2d',d'}Z_2=(C^T\otimes I_{n_2})Q_2I_{n_2d,d}$.
Hence, $Z_2$ is the upper-left $d'\times d$-submatrix of
${Q'_2}^{-1}(C^T\otimes I_{n_2})Q_2=
P_{\tau'}^{-1}((W^{-1}CW')^T\otimes I_{n_2})P_\tau$.

Let $W^{-1}CW'=\{\hat c_{l,l'}\}_{1\le l\le d;1\le l'\le d'}$,
$Z_2=\{\tilde c_{l',l}\}_{1\le l'\le d';1\le l\le d}$. Then
direct computation shows that
$\tilde c_{l',l}=\hat c_{l,l'}\delta_{i_l}^{i'_{l'}}$.
Since $(W^{-1}CW')(W'^{-1}\chi'(\sigma_2)W')=
(W^{-1}\chi(\sigma_2)W)(W^{-1}CW')$, we obtain
that if $i_l\neq i'_{l'}$, then $\hat c_{l,l'}=0$.
It yields $Z_2=(W^{-1}CW')^T$.
Since $i_l=i'_{l'}=0$ for $1\le l\le d_s$,
$1\le l'\le d'_s $ and $i_l\neq0$, $i'_{l'}\neq0$, for
$d_s+1\le l\le d$, $d'_s +1\le l'\le d'$, we obtain
that $Z_2$ is a block-diagonal matrix with two blocks of size
$d'_s \times d_s$ and $d'_a\times d_a$.
One can easily see that the upper-left block is
$(W_s^{-1}C_sW'_s)^T$ and the lower-right block is
$(W_a^{-1}C_aW'_a)^T$.
Finally, the linear coefficient of the composition
$F_3\stackrel{f_3}{\longrightarrow}F_2\longrightarrow F'_2
\stackrel{(f'_3)^{-1}}{\longrightarrow}F'_3$, is equal to
$$
\begin{pmatrix}
(W'_s)^T & 0\\
0 & (W'_a)^T
\end{pmatrix}^{-1}Z_2
\begin{pmatrix}
W_s^T & 0\\
0 & W_a^T
\end{pmatrix}=
\begin{pmatrix}
C_s^T & 0\\
0 & C_a^T
\end{pmatrix}
$$
as required.
\end{proof}

\begin{tpt3+}\label{tpt3+}
Suppose that the assumptions of Theorem~\ref{tpt3} are satisfied.

{\rm I.} $\hat\cat$ is isomorphic to the direct sum
of a $p$-divisible group and $d_a$ copies of the additive
formal group scheme.

{\rm II.} If $K/\Q_p$ is unramified, then $\hat\cat$
is represented by a formal group law whose logarithm is of type
$pI_d-U_1\bt$.

{\rm III.} If $K/\Q_p$ is totally tamely ramified, then
$\hat\cat$ is isomorphic to the direct sum of $d_s$ copies
of the multiplicative formal group scheme and $d_a$ copies
of the additive formal group scheme.
\end{tpt3+}

Part II of the above Corollary follows from Theorem 1.5 of
\cite{DN}. Besides, Theorems 0.1 and 1.3 of \cite{NX} imply that
part III holds for the reduction of the formal group scheme
$\hat\cat$.

\begin{isrefg}\label{isrefg}
Let $F$, $F'$ be $d$-dimensional formal group laws over $\Z_p$
with logarithms of types $u$, $u'$,
respectively, such that their reductions are isomorphic.
If
$$
u=pI_d-\begin{pmatrix}
S & 0\\
0 & 0
\end{pmatrix}\bt, \quad
u'=pI_d-\begin{pmatrix}
S' & 0\\
0 & 0
\end{pmatrix}\bt,
$$
with matrices $S\in GL_{d_1}(\Z_p)$, $S'\in GL_{d'_1 }(\Z_p)$,
then $F$ and $F'$ are isomorphic.
\end{isrefg}

\begin{proof}
By Theorem $6$ of \cite{Ho}, there exist $w,z\in\GL_d(\cae)$ such
that $u'z=wu$. In particular, it implies $d'_1 =d_1 $. Denote by
$\tilde u$, $\tilde u'$, $\tilde z$, $\tilde w$ the  upper-left
$d_1\times d_1$-submatrix of $u$, $u'$, $z$, $w$, respectively.
Then we get $\tilde u'\tilde z=\tilde w\tilde u$. Let $\lambda$, $\lambda'$
be of types $\tilde u$, $\tilde u'$, respectively.
By Proposition~\ref{hon} (i), $\lambda$ and $\lambda'$ are the logarithms
of formal group laws $\tilde F$ and $\tilde F'$. Obviously,
$F$ is isomorphic to the direct sum of $\tilde F$ and $d-d_1$ copies
of the additive formal group law, and $F'$ is isomorphic to the direct sum
of $\tilde F'$ and $d-d_1$ copies of the additive formal group law.
Thus it remains to show that $\tilde F$ is isomorphic to $\tilde F'$. Let
$$
z=
\begin{pmatrix}
\tilde z & \hat z\\
* & *
\end{pmatrix},
$$
with $\hat z\in\M_{d_1,d-d_1}(\cae)$. Since
$$
u'z=wu\equiv
\begin{pmatrix}
* & 0\\
* & 0
\end{pmatrix}\mod p,
$$
we get $\tilde u'\hat z\equiv 0\mod p$. Therefore $\hat z\equiv 0\mod p$,
and since $z$ is invertible, $\tilde z $ is also invertible.
Take $D\in\GL_{d_1 }(\Z_p)$ such that $\tilde z=D\mod\bt$.
Then $\lambda''=D^{-1}\tilde z\lambda$ is of type $\tilde u\tilde z^{-1}D$,
and hence, by Proposition~\ref{hon} (i),(iii),
the formal power series $\lambda''$ is
the logarithm of a formal group law $\tilde F''$ over $\Z_p$
which is isomorphic to $\tilde F'$.  Since
$(\lambda^{-1}\circ \tilde z^{-1}D\lambda'')(X)=X$,
Theorem $5$ (ii) of \cite{Ho} implies that the reduction
of $\tilde F$ coincides with that of $\tilde F''$.
Since $\tilde u=pI_{d_1 }-S\bt$
and $S$ is invertible, $\tilde F$ is a $p$-divisible group of
height $d_1 $. Then by Theorem 2 of \cite{DG}, any
deformation over $\Z_p$ of the reduction of $\tilde F$
is isomorphic to $\tilde F$. In particular, $\tilde F''$
is isomorphic to $\tilde F$.
\end{proof}

\begin{isreto}\label{isreto}
Let $T$, $T'$ be tori over $\Q_p$ which are split
over an abelian tamely ramified extension of $\Q_p$, and
$\cat$, $\cat'$ be their N\'eron models. If the reductions
of $\cat$ and $\cat'$ are isomorphic, then $\hat\cat$
and $\hat\cat'$ are isomorphic.
\end{isreto}

\begin{proof}
It is clear that the reductions of $\hat\cat$ and
$\hat\cat'$ are isomorphic. Then by Theorem~\ref{tpt3},
$\hat\cat$ and $\hat\cat'$ are represented by formal
group laws which satisfy the conditions of Lemma~\ref{isrefg}.
Therefore $\hat\cat$ and $\hat\cat'$ are isomorphic.
\end{proof}

\noindent{\bf Remark.} It is easy to give an example
of tori $T$, $T'$ such that $\hat\cat$ and $\hat\cat'$
are isomorphic, but the reductions of $\cat$ and $\cat'$
are not. Take $K$ to be the unramified extension of $\Q_p$
of degree 2, and define
$\chi,\chi'\colon\Gal(K/\Q_p)\to\GL_2(\Z)$ as follows:
$$
\chi(\Delta_p)=\begin{pmatrix}
1 & 0\\
0 & -1
\end{pmatrix}, \quad
\chi'(\Delta_p)=\begin{pmatrix}
0 & 1\\
1 & 0
\end{pmatrix}.
$$
Let $T$, $T'$ be tori over $\Q_p$ corresponding to
the representations $\chi,\chi'$. Then the connected
components of the reductions of $\cat$, $\cat'$ are
tori over $\F_p$ which correspond to representations
$\bar\chi,\bar\chi'\colon\Gal(\F_{p^2}/\F_p)\to\GL_2(\Z)$
defined by
$$
\bar\chi(\Delta_p)=\begin{pmatrix}
1 & 0\\
0 & -1
\end{pmatrix}, \quad
\bar\chi'(\Delta_p)=\begin{pmatrix}
0 & 1\\
1 & 0
\end{pmatrix}.
$$
Clearly, these representations are not isomorphic, and
hence, the reductions of $\cat$ and $\cat'$ are not
isomorphic. On the other hand, by Theorem~\ref{tpt3},
$\hat\cat$, $\hat\cat'$ can be represented by formal
group laws with logarithms of types
$u=pI_2-\chi(\Delta_p)\bt$,
$u'=pI_2-\chi'(\Delta_p)\bt$, respectively.
Since
$$
u\begin{pmatrix}
1 & 1\\
-1 & 1
\end{pmatrix}=\begin{pmatrix}
1 & 1\\
-1 & 1
\end{pmatrix}u',
$$
Proposition~\ref{hon} (iii) implies that for $p\neq2$,
$\hat\cat$ and $\hat\cat'$ are isomorphic.

\section{One-dimensional tori over $\Q$}

We keep the notations of Section 6. Suppose that
$T$ is a non-split one-dimensional torus over $L=\Q$
which is split over a tamely ramified quadratic extension $K$ of $\Q$.
Then $d=1$, $n=2$, $\Gal(K/\Q)=\{\id_{\Q},\sigma\}$ and
the discriminant of $\cao_K$ is odd, i.e., there exists
$\xi\in\cao_K$ such that $1,\xi$ form a free $\Z$-basis
of $\cao_K$ and $q=(\sigma(\xi)-\xi)^2$ is odd.
Let $x^2-rx+s\in\Z[X]$ be the minimal polynomial for $\xi$.
Then $q=r^2-4s$. Finally, put $e_0=1$, $e_1=\xi$.

\begin{tip41}\label{tip41}
$\Lambda$ is of $p$-type $v_p=pI_2-V_p\bt_p$, where
$$
V_p=\left\{
\begin{array}{rl}
I_2, & {\it if} \quad\bigl(\frac qp\bigr)=1\\
\begin{pmatrix}
1 & r\\
0 & -1
\end{pmatrix}, & {\it if} \quad\bigl(\frac qp\bigr) =-1\\
\begin{pmatrix}
1 & r/2\\
0 & 0
\end{pmatrix}, & {\it if} \quad p\mid q
\end{array}\right..
$$
\end{tip41}

\begin{proof}
Let $(p,q)=1$. Then $\Q_p(\xi)/\Q_p$ is an unramified
extension of degree 1, if $\bigl(\frac qp\bigr)=1$, and 2,
if $\bigl(\frac qp\bigr) =-1$. Indeed, for $p\neq2$, it
follows from the definition of the Jacobi symbol, and
for $p=2$, it is true, since $q\equiv 1\mod 4$.

Denote $\phi(z_1,z_2)=z_1+\xi z_2$.
By Corollary from Proposition~\ref{wr_mgs}, we have
$\Lambda_1+\Lambda_2\xi=\LL \circ\phi$.

We consider the action of $\bt_p$ on $\Q_p(\xi)[[z_1,z_2]]$
introduced in Section~2. Then according to Lemma~\ref{ass},
$$(p-\bt_p)(\Lambda_1+\Lambda_2\xi)\equiv
((p-\bt_p)\LL )\circ\phi\equiv0\mod p.$$

If $\bigl(\frac qp\bigr)=1$, then $\bt_p(\xi)=\xi$, and
we get $p\Lambda_1-\bt_p\Lambda_1\equiv0\mod p$
and $p\Lambda_2-\bt_p\Lambda_2\equiv0\mod p$.
If $\bigl(\frac qp\bigr)=-1$, then $\bt_p(\xi)=r-\xi$, and we obtain
$p\Lambda_1-\bt_p(\Lambda_1+r\Lambda_2)\equiv0\mod p$
and $p\Lambda_2+\bt_p\Lambda_2\equiv0\mod p$.
Thus in each case, $\Lambda$ is of the required $p$-type.

We proceed with the case $p\mid q$. Since $1, \xi$ is a free
$\Z$-basis in $\cao_K$, $p^2\nmid q$.
Let $\pi=r-2\xi$. We have $\pi^2=q$. Therefore, $\Q_p(\xi)/\Q_p$
is a totally ramified extension of degree 2, and $e_0,e_1$ form
a free $\Z_p$-basis of $\Z_p[\xi]$. Then Proposition~\ref{wr_mgs}
implies that $\Phi=\car_{\Z_p[\xi]/\Z_p}((\F_m)^d_{\Z_p[\xi]})$.
Take $e'_0=1$, $e'_1=\pi$, and denote
$\Phi'=\car'_{\Z_p[\xi]/\Z_p}((\F_m)^d_{\Z_p[\xi]})$.
Let $\Lambda'$ be the logarithm of $\Phi'$.
Then by Proposition~\ref{tran_wr}, we get
$\Lambda^{-1}\circ (I_d\otimes W)\Lambda'
\in\Hom_{\Z_p}(\Phi',\Phi)$, where
$$W=\begin{pmatrix}
1 & r\\
0 & -2
\end{pmatrix}.
$$
Applying Proposition~\ref{tip3} for the data $d=1$, $n_1=1$,
$n_2=2$, we deduce that $\Lambda'$ is of $p$-type $pI_2-J_2\bt_p$.
Hence, Proposition~\ref{hon} (iii) implies that $\Lambda$ is
of $p$-type $pI_2-WJ_2W^{-1}\bt_p$ just as required.
\end{proof}

Since  $\sigma(\xi)=r-\xi$, in the basis $e_0$, $e_1$
the automorphism $\sigma$ is given by the matrix
$\begin{pmatrix}
1 & r\\
0 & -1
\end{pmatrix}$.
Therefore, $\psi(\sigma)=
\begin{pmatrix}
1 & 0\\
r & -1
\end{pmatrix}$. Since $T$ is non-split, $\chi(\sigma)=-1$.
Then
$\theta(\sigma)^T=
\begin{pmatrix}
-1 & -r\\
0 & 1
\end{pmatrix}$ by Theorem~\ref{q_mat}.

For a prime number $p$, let $\Xi(p)=\bigl(\frac qp\bigr)$,
if $(p,q)=1$; $\Xi(p)=0$, if $p\mid q$.
Further, let $F_\Xi$ be defined as at the end
of Section 2. Due to Proposition~\ref{glob},
$F_\Xi$ is a formal group law over $\Z$.
Finally, denote the formal group laws
$F_{r,s}(x,y)=(x+y+rxy)(1-sxy)^{-1}$
and $F_q(x,y)=x+y+\sqrt q xy$.

\begin{tpt4}\label{tpt4}
Let $T$ be a non-split one-dimensional torus over $\Q$
which is split over a tamely ramified quadratic
extension of $\Q$, and $\cat$ be its N\'eron model. Then

{\rm I.} $F_\Xi$ represents $\hat\cat$;

{\rm II.} $F_\Xi$ is strongly isomorphic
to $F_{r,s}$ over $\Z$;

{\rm III.} $F_\Xi$ is strongly isomorphic
to $F_q$ over $\Z[\xi]$.

\end{tpt4}

\begin{proof}
According to Theorem~\ref{repr}, a formal group law from
a universal fixed pair for $(\Phi,\Gal(\Q(\xi)/\Q))$ represents
$\hat\cat$. Take
$$
Q=\begin{pmatrix}
-r & (r+1)/2\\
2 & -1
\end{pmatrix}.
$$
Then condition (iii) of Theorem~\ref{ufix_ex},I is obviously
satisfied for $L=\Q_p$ and any prime $p$. Further, applying
Proposition~\ref{tip41} one can directly check that the upper-left
entry of $Q^{-1}v_pQ$ is equal to $pI_d-\Xi(p)\bt_p$.
On the other hand, Proposition~\ref{glob} implies that
the logarithm of $F_\Xi$ is of $p$-type $pI_d-\Xi(p)\bt_p$.
Then by Proposition~\ref{glufix}, there exists
$f_\Xi\in\Hom_{\Z}(F_\Xi,\Phi)$ such that $(F_\Xi,f_\Xi)$ is
a universal fixed pair for $(\Phi,\Gal(\Q(\xi)/\Q))$. Thus $F_\Xi$
represents $\hat\cat$.

Proposition~\ref{wr_mgs} implies that $\Phi=(\Phi_1,\Phi_2)$,
where $\Phi_1(x_1,x_2;y_1,y_2)=x_1+y_1+x_1y_1-sx_2y_2,$
$\Phi_2(x_1,x_2;y_1,y_2)=x_2+y_2+x_1y_2+x_2y_1+rx_2y_2.$
Besides, one can check that $g(x_1,x_2)=x_2(1+x_1)^{-1}$
belongs to $\Hom_{\Z}(\Phi,F_{r,s})$ and its linear
coefficient is $(0,1)\in\M_{1,2}(\Z)$. According to
Proposition~\ref{hon} (ii), for every prime $p$,
there exists $u'_p\in\cae_p$ such that the logarithm
$\lambda_{r,s}$ of $F_{r,s}$ is of $p$-type $u'_p$.
By Proposition~\ref{hon} (iii), we have
$u'_p(0,1)=(\hat w_p,\tilde w_p)v_p$ for some
$\hat w_p,\tilde w_p\in\cae_p$.
Applying Proposition~\ref{tip41} we obtain $\hat w_p=0$
and $u'_p=\tilde w_p(p-\Xi(p)\bt_p)$. Therefore,
$\tilde w_p\equiv1\mod\bt_p$, and $\lambda_{r,s}$ is
of $p$-type $\tilde w_p^{-1}u'_p=p-\Xi(p)\bt_p$.
Further, Proposition~\ref{hon} (iii) implies that
$\lambda_{r,s}^{-1}\circ\lambda_{\Xi}
\in\Hom_{\Z_p}(F_{\Xi},F_{r,s})$.
Since it is true for any prime $p$,
$\lambda_{r,s}^{-1}\circ\lambda_{\Xi}
\in\Hom_{\Z}(F_{\Xi},F_{r,s})$, i.e., $ F_{\Xi}$
and $F_{r,s}$ are strongly isomorphic over $\Z$.

Finally, if $\sqrt q=r-2\xi$, then
$x(1+\xi x)^{-1}\in\Hom_{\Z[\xi]}(F_{r,s},F_q)$.
\end{proof}

Part I of the above Proposition is a special case
of Theorem 1.5 of \cite{DN}, and part III is
precisely Theorem 4 of \cite{Ho1}.

\section{Tori split over a tamely ramified abelian extension of $\Q$}

We keep the notation of Section~6. Suppose that $L=\Q$ and $K/\Q$
is a tamely ramified abelian extension. Due to the Kronecker-Weber
Theorem, one can replace $K$ by a larger field so that
$K=\Q(\xi)$, where $\xi$ is a primitive $q$-th root of unity and
$q=p_1\cdots p_k$ is the product of distinct primes $p_i$. Then
$K$ is the composite of the extensions $\Q(\xi_i)/\Q$, $1\leq i
\leq k$, where $\xi_i=\xi^{q/p_i}$, and the degree of $K/\Q$ is
$n=(p_1-1)\cdots(p_k-1)$. Denote
$K_i=\Q_{p_i}(\xi_1,\ldots,\xi_{i-1},\xi_{i+1},\ldots,\xi_k)$.

For any $i=1,\dots,k$, fix $s_i\in\Z$ such that $s_i$ is a
multiplicative generator modulo $p_i$, i.e. $s_i\not\equiv 0\mod
p_i$ and $s_i^\alpha\not\equiv 1\mod p_i$ for any $1\leq
\alpha\leq p_i-2$. Further, for any integer $l$ relatively prime
to $p_i$, take $0\leq r_i(l)\leq p_i-2$ such that $l\equiv
s_i^{r_i(l)}\mod p_i$. Denote also
$n_i=\prod_{j=1}^{i-1}(p_j-1)$ for $1\leq i\leq k$.
Consider the bijection
$$
\gamma\colon\prod_{i=1}^k\{0,\ldots,p_i-2\}\to\{0,\ldots,n-1\}
$$
given by
$ \gamma(\alpha_1,\ldots,\alpha_k)= \sum_{i=1}^k\alpha_i n_i$.
Notice that for any matrices
$U^{(i)}=\{a^{(i)}_{\alpha_i,\beta_i}\}_
{0\leq\alpha_i,\beta_i\leq p_i-2}, 1\leq i\leq k$, and
$U^{(1)}\otimes\cdots\otimes U^{(k)}=\{a_{s,t}\}_{0\leq s,t\leq
n-1}$, we have
$a_{\gamma(\alpha_1,\ldots,\alpha_k),\gamma(\beta_1,\ldots,\beta_k)}=
a^{(1)}_{\alpha_1,\beta_1}\cdots a^{(k)}_{\alpha_k,\beta_k}$.

For $0\leq \alpha_i\leq p_i-2$, $1\le i\le k$, put
$e_{\gamma(\alpha_1,\ldots,\alpha_k)}=\prod_{i=1}^k\xi_i^{s_i^{\alpha_i}}$.
Obviously, $e_0,\ldots,e_{n-1}$ is a free $\Z$-basis of $\cao_K=\Z[\xi]$.

\begin{tip51}\label{tip51}
$\Lambda$ is of $p$-type $v_p=pI_{nd}-I_d\otimes V_p\bt_p$, where
$V_p=P_{p_1-1}^{r_1(p)}\otimes\cdots\otimes P_{p_k-1}^{r_k(p)}$ if
$p\ne p_i$ for any $i=1,\dots,k$ and
$$
V_{p_i}=P_{p_1-1}^{r_1(p_i)}\otimes\cdots \otimes
P_{p_{i-1}-1}^{r_{i-1}(p_i)}
\otimes(p_iI_{p_i-1}-J'_{p_i-1})\otimes
P_{p_{i+1}-1}^{r_{i+1}(p_i)}\otimes \cdots\otimes
P_{p_k-1}^{r_k(p_i)}.
$$
\end{tip51}

\begin{proof}
Let $p\ne p_i$ for any $1\le i\le k$. Then $\Q_p(\xi)/\Q_p$
is unramified. For any $1\leq l\leq d$, consider
$\phi_l=\sum_{i=0}^{n-1}z_{id+l}e_i$.
By Corollary from Proposition~\ref{wr_mgs}, we have
$\sum_{j=0}^{n-1}\Lambda_{jd+l}e_j=\LL \circ\phi_l$
for any $1\leq l\leq d$.

We consider the action of $\bt_p$ on $\Q_p(\xi)[[z_1,\ldots,z_{nd}]]$
introduced in Section 2. Then applying Lemma~\ref{ass} we get
$$
(p-\bt_p) \sum_{j=0}^{n-1}\Lambda_{jd+l}e_j\equiv
((p-\bt_p)\LL )\circ\phi_l\equiv0\mod p.
$$
Hence, $\sum_{j=0}^{n-1}p\Lambda_{jd+l}e_j\equiv
\sum_{j=0}^{n-1}(\bt_p\Lambda_{jd+l})e_j^p\mod p$.
If $0\leq \alpha_i,\alpha'_i\leq p_i-2$, $1\le i\le k$,
are such that $\alpha'_i\equiv\alpha_i+r_i(p)\mod p_i-1$,
then $e_{\gamma(\alpha_1,\ldots,\alpha_k)}^p=
e_{\gamma(\alpha'_1,\ldots,\alpha'_k)}$. Therefore, we obtain
$$
p\Lambda_{\gamma(\alpha'_1,\ldots,\alpha'_k)d+l}-
\bt_p\Lambda_{\gamma(\alpha_1,\ldots,\alpha_k)d+l}\equiv0\mod p.
$$
Thus $\Lambda$ is of the required $p$-type.

We proceed with the case $p=p_i$.
Since $e_0,\ldots,e_{n-1}$ is a free $\Z_{p_i}$-basis of $\Z_{p_i}[\xi]$,
Proposition~\ref{wr_mgs} implies that
$\Phi=\car_{\Z_{p_i}[\xi]/\Z_{p_i}}((\F_m)^d_{\Z_{p_i}[\xi]})$.
Take $\pi_i$ such that $\pi_i^{p_i-1}=-p_i$. Then by Lubin-Tate theory,
$\Q_{p_i}(\xi_i)=\Q_{p_i}(\pi_i)$ and there exists
$W=\{w_{\alpha,\beta}\}_{0\leq \alpha,\beta\leq p_i-2}\in\GL_{p_i-1}(\Z_{p_i})$
such that
$\pi_i^\beta=\sum_{\alpha=0}^{p_i-2}w_{\alpha,\beta}\xi_i^{s_i^\alpha}$
for any $0\le\beta\le p_i-2$.

Take $ e'_0,\ldots,  e'_{n-1}$, another free $\Z_{p_i}$-basis of $\Z_{p_i}[\xi]$,
where
$$
e'_{\gamma(\alpha_1,\ldots,\alpha_k)}=\pi_i^{\alpha_i}
\prod_{1\le j\le k,\, j\neq i}
\xi_j^{s_j^{\alpha_j}},\qquad
0\leq \alpha_i\leq p_i-2, 1\le i\le k.
$$
Denote for this basis
$\Phi'=\car'_{\Z_p[\xi]/\Z_p}((\F_m)^d_{\Z_p[\xi]})$.
Let $\Lambda'$ be the logarithm of $\Phi'$,
then Proposition~\ref{tran_wr} implies that
$\Lambda^{-1}\circ (I_d\otimes \hat W)\Lambda'
\in\Hom_{\Z_p}(\Phi',\Phi)$, where
$$
\hat W=I_{p_1-1}\otimes\cdots\otimes I_{p_{i-1}-1}\otimes
W\otimes I_{p_{i+1}-1}\otimes\cdots\otimes I_{p_k-1}.
$$
For any $1\leq l\leq d$, consider
$\phi'_l=\sum_{0\leq i\le n-1}z_{id+l} e'_i$.
By Corollary from Proposition~\ref{wr_mgs}, we have
$\sum_{0\leq i\leq n-1} \Lambda'_{id+l} e'_i=\LL \circ\phi'_l$.

Remind that for any $\kappa\in\Q_{p_i}(\xi)[[y_1,\ldots,y_m]]$,
there exist unique
$\kappa^{(j)}\in K_i[[y_1,\ldots,y_m]]$, $0\leq j\leq p_i-2$,
such that $\kappa=\sum_{j=0}^{p_i-2}\kappa^{(j)}\pi_i^j$.
Thus for any $0\leq\alpha_i\leq p_i-2$,
$$
\sum_{\scriptstyle 0\leq \alpha_j\leq p_j-2\atop\scriptstyle j\neq i}
\Lambda'_{\gamma(\alpha_1,\ldots,\alpha_k)d+l} \prod_{j\neq i}
\xi_j^{s_j^{\alpha_j}}=\left(\LL \circ\phi'_l\right)^
{(\alpha_i)}.
$$
If $\alpha_i\neq 0$, then by Lemma~\ref{sm_ram}
$$
\sum_{\scriptstyle 0\leq \alpha_j\leq p_j-2\atop\scriptstyle j\neq
i}\Lambda'_{\gamma(\alpha_1,\ldots,\alpha_k)d+l} \prod_{j\neq
i}\xi_j^{s_j^{\alpha_j}}\equiv
\LL ^{(\alpha_i)}\circ\phi_l^{\prime(0)}=0\mod\Z_{p_i}[\xi],
$$
whence
$p_i\Lambda'_{\gamma(\alpha_1,\ldots,\alpha_k)d+l}\equiv0\mod
p_i$.

If $\alpha_i=0$, then Lemma~\ref{sm_ram} implies
$$
\sum_{\scriptstyle 0\leq \alpha_j\leq p_j-2\atop\scriptstyle j\neq
i}\Lambda'_{\gamma(\alpha_1,\ldots,\alpha_k)d+l} \prod_{j\neq
i}\xi_j^{s_j^{\alpha_j}}\equiv
\LL \circ\phi_l^{\prime(0)}\mod p_i.
$$
We consider the action of $\bt_{p_i}$ on $K_i[[z_1,\ldots,z_{nd}]]$
introduced in Section 2. Then applying Lemma~\ref{ass} we get
$$
(p_i-\bt_{p_i})\sum_{\scriptstyle 0\leq \alpha_j\leq
p_j-2\atop\scriptstyle j\neq
i}\Lambda'_{\gamma(\alpha_1,\ldots,\alpha_k)d+l} \prod_{j\neq
i}\xi_j^{s_j^{\alpha_j}}\equiv
(p_i-\bt_{p_i})\left(\LL \circ\phi_l^{\prime(0)}\right)\equiv
$$
$$
\left((p_i-\bt_{p_i})\LL \right)\circ\phi_l^{\prime(0)}\equiv
0\mod p_i.
$$
Therefore,
$$
p_i\!\!\!\!\!\sum_{\scriptstyle 0\leq \alpha_j\leq
p_j-2\atop\scriptstyle j\neq
i}\Lambda'_{\gamma(\alpha_1,\ldots,\alpha_k)d+l} \prod_{j\neq
i}\xi_j^{s_j^{\alpha_j}}\equiv \!\!\!\!\!\!\!\sum_{\scriptstyle
0\leq \alpha_j\leq p_j-2\atop\scriptstyle j\neq
i}\!\!\!\!\bt_{p_i}\Lambda'_{\gamma(\alpha_1,\ldots,\alpha_k)d+l}
\prod_{j\neq i}\xi_j^{s_j^{\alpha_j+r_j(p_i)}}\!\!\!\!\mod p_i.
$$
As before, choose $0\leq\alpha'_j\leq p_j-2$ for $j=1,\ldots,i-1,i+1,\ldots,k$
such that $\alpha'_j\equiv\alpha_j+r_j(p_i)\mod p_j-1$ and put $\alpha_i=0$.
It gives
$$
p_i\Lambda'_{\gamma(\alpha'_1,\ldots,\alpha'_k)d+l}\equiv
\bt_{p_i}\Lambda'_{\gamma(\alpha_1,\ldots,\alpha_k)d+l}\mod
p_i.
$$
Thus $\Lambda'$ is of a $p_i$-type
$v'=p_iI_{nd}-I_d\otimes V'\bt_{p_i}$, where
$$
V'=P_{p_1-1}^{r_1(p_i)}\otimes\cdots
\otimes P_{p_{i-1}-1}^{r_{i-1}(p_i)}
\otimes J_{p_i-1}\otimes
P_{p_{i+1}-1}^{r_{i+1}(p_i)}\otimes
\cdots\otimes P_{p_k-1}^{r_k(p_i)}.
$$
Hence, by Proposition~\ref{hon} (iii), $\Lambda$ is of $p_i$-type
$v=(I_d\otimes\hat W)v'(I_d\otimes\hat W)^{-1}=
p_iI_{nd}-I_d\otimes V\bt_{p_i}$, where $V=\hat WV'\hat W^{-1}$.

Since $1=\sum_{\alpha=0}^{p_i-2}(-\xi_i^{s_i^\alpha})$, we get
$w_{\alpha,0}=-1$ for any $0\leq \alpha\leq p_i-2$, i.e., all
entries of the first column of $W$ are equal to $-1$. Further, let
$W^{-1}=\{w'_{\alpha,\beta}\}_{0\leq \alpha,\beta\leq p_i-2}$.
For any $0\leq\beta\leq p_i-2$, there exists
$\tau_\beta\in\Gal(\Q_{p_i}(\xi_i)/\Q_{p_i})$
such that $\tau_\beta(\xi_i)=\xi_i^{s_i^\beta}$.
Denote $\zeta_\beta=\tau_\beta(\pi_i)/\pi_i$.
Then $\zeta_\beta^{p_i-1}=1$, $\zeta_\beta\in\Z_{p_i}$,
and hence, $\tau_\beta(\xi_i)=\sum_{\alpha=0}^{p_i-2}
w'_{\alpha,0}\tau_\beta(\pi_i)^\alpha=\sum_{\alpha=0}^{p_i-2}
w'_{\alpha,0}\zeta_\beta^\alpha\pi_i^\alpha$ which implies
$w'_{\alpha,\beta}=\zeta_\beta^\alpha w'_{\alpha,0}$.
Since $\sum_{\beta=0}^{p_i-2}\zeta_\beta^\alpha=0$
for any $1\le\alpha\le p_i-2$, we get
$$
-1=\sum_{\beta=0}^{p_i-2}\xi_i^{s_i^\beta}=
\sum_{\alpha=0}^{p_i-2}\sum_{\beta=0}^{p_i-2}
\zeta_\beta^\alpha w'_{\alpha,0}\pi_i^\alpha=(p_i-1)w'_{0,0},
$$
and then $w'_{0,\beta}=w'_{0,0}=-(p_i-1)^{-1}$, i.e., all entries of
the first row of $W^{-1}$ are equal to $-(p_i-1)^{-1}$. Therefore,
$WJ_{p_i-1}W^{-1}=(p_i-1)^{-1}J'_{p_i-1}$, and thus
$$
V=P_{p_1-1}^{r_1(p_i)}\otimes\cdots \otimes
P_{p_{i-1}-1}^{r_{i-1}(p_i)} \otimes (p_i-1)^{-1}J'_{p_i-1}\otimes
P_{p_{i+1}-1}^{r_{i+1}(p_i)}\otimes \cdots\otimes
P_{p_k-1}^{r_k(p_i)}.
$$
Clearly, $\Lambda$ is also of $p_i$-type
$(I_{nd}-I_d\otimes W'\bt_{p_i})v$, where
$$
W'=P_{p_1-1}^{r_1(p_i)}\otimes\cdots \otimes
P_{p_{i-1}-1}^{r_{i-1}(p_i)}\otimes
\left(I_{p_i-1}-(p_i-1)^{-1}J'_{p_i-1}\right) \otimes
P_{p_{i+1}-1}^{r_{i+1}(p_i)}\otimes\cdots \otimes
P_{p_k-1}^{r_k(p_i)}.
$$
Finally, since
$\left(I_{p_i-1}-(p_i-1)^{-1}J'_{p_i-1}\right)J'_{p_i-1}=0$, we
obtain that $\Lambda$ is of the  required $p_i$-type.
\end{proof}

Let $\sigma_i\in\Gal(\Q(\xi)/\Q), 1\leq i\leq k$,
be such that $\sigma_i(\xi_i)=\xi_i^{s_i}$,
$\sigma_i(\xi_j)=\xi_j$, for $j\neq i$.
They span $\Gal(\Q(\xi)/\Q)$ and,
in the basis $e_0,\ldots,e_{n-1}$,
are given by the matrices
$$
\hat P_i=I_{p_1-1}\otimes\cdots\otimes I_{p_{i-1}-1}\otimes
P_{p_i-1}\otimes I_{p_{i+1}-1}\otimes\cdots\otimes I_{p_k-1}.
$$
Therefore, $\psi(\sigma_i)=(\hat P_i^{-1})^T=\hat P_i$.
Denote $U_i=\chi(\sigma_i)^T$. Then $U_i U_j=U_j U_i$
and by Theorem~\ref{q_mat},
$\theta(\sigma_i)^T=U_i\otimes\hat P_i^T$.

From now on, we employ the ``matrix-of-matrices'' representation
described in Introduction. Denote $D_i=U_i\otimes\hat
P_i^T-I_{nd}$. Then, under our convention,
$D_i=\{d^{(i)}_{s,t}\}\in\M_n(\M_d(\Z))$, where $ d^{(i)}_{s,t}=
\delta_{s}^{t-n_i}U_i-\delta_{s}^{t}I_d$ for
$n_i=\prod_{j=1}^{i-1}(p_j-1)=\gamma(\epsilon_i)$, and $\epsilon_i$ stands for
the $k$-vector $(0,\ldots,1,\ldots,0)$ with unity on the $i$-th position.

\begin{ms_pabz}\label{ms_pabz}
Denote $U^{\langle\alpha\rangle}=U_1^{-\alpha_1}\cdots U_k^{-\alpha_k}$
for $0\leq\alpha_i\leq p_i-2$
and let $Q=\{q_{s,t}\}\in\GL_n(\M_d(\Z))$, where
$$
q_{\gamma(\alpha), \gamma(\beta)} =\left\{\begin{array}{rl}
U^{\langle\alpha\rangle}, &
{\it if} \quad \gamma(\beta) = 0 \quad {\it and} \quad
\gamma(\alpha)\ne 0\\
\delta_{\gamma(\alpha)}^{\gamma(\beta)}I_d, & {\it otherwise}
\end{array}\right..
$$
Then for any $1\le i\le k$,
$$
Q^{-1}D_iQ=
\begin{pmatrix}
0 & \hat D_i\\
0 & \tilde D_i
\end{pmatrix},
$$
where $\hat D_i\in\M_{1,n-1}(\M_d(\Z))$
and $\tilde D_i\in\M_{n-1}(\M_d(\Z))$.

Moreover, for any prime $p$ there exist
$\hat C^{(p)}_i\in\M_{n-1,1}(\M_d(\Z_p))$ and
$\tilde C^{(p)}_i\in\M_{n-1}(\M_d(\Z_p))$ such that
$\sum_{i=1}^k\hat C^{(p)}_i\hat D_i+\tilde C^{(p)}_i\tilde D_i
\in\GL_{n-1}(\M_d(\Z_p))$.
\end{ms_pabz}

\begin{proof}
Observe first that $Q^{-1}=\{q'_{s,t}\}\in\GL_n(\M_d(\Z))$ is the
matrix given by $q'_{s,t}=-q_{s,t}$ if $s=0$ and $t\ne 0$, and
$q'_{s,t}=q_{s,t}$ otherwise.

We have $Q^{-1}D_iQ=\{a^{(i)}_{s,t}\}_{0\leq s,t\leq n-1}$, where
$$
a^{(i)}_{s,t}=\sum_{s',t'=0}^{n-1} q'_{s,s'}d^{(i)}_{s',t'}q_{t',t}.
$$
Direct calculation shows that
$$
a^{(i)}_{0,0}=-I_d+U_iU^{\langle \epsilon_i\rangle}=-I_d+U_iU_i^{-1}=0
$$
and
$$
a^{(i)}_{\gamma(\alpha),0}=U^{\langle\alpha\rangle}-
U^{\langle\alpha\rangle}U_iU^{\langle \epsilon_i\rangle}=0
$$
if $\gamma(\alpha)\ne 0$, proving the first assertion.

In order to prove the existence of $\hat C^{(p)}_i$ and
$\tilde C^{(p)}_i$, we need to compute explicitly the matrices
$\hat D_i=\{\hat a^{(i)}_t\}_{1\leq t\leq n-1}$ and
$\tilde D_i=\{a^{(i)}_{s,t}\}_{1\leq s,t\leq n-1}$. Here we get
$$
\hat a^{(i)}_t=a^{(i)}_{0,t}=\delta_{n_i}^{t}U_i
$$
and
$$
a^{(i)}_{\gamma(\alpha),\gamma(\beta)}=
-\delta^{\gamma(\beta)}_{n_i}U_iU^{\langle\alpha\rangle}
+ \delta_{\gamma(\alpha)+n_i}^{\gamma(\beta)}U_i-
\delta_{\gamma(\alpha)}^{\gamma(\beta)}I_d
$$
for $\gamma(\alpha)\ne 0$ and $\gamma(\beta)\ne 0$.

By Lemma~\ref{ker_zero}, we only need to prove that, for any prime $p$
$$
\left(\cap_{i=1}^k\Kera (\hat D_i\otimes \F_p)\right)\cap
\left(\cap_{i=1}^k\Kera (\tilde D_i\otimes \F_p)\right)=\{0\}.
$$
Let $(x_1,\ldots, x_{n-1})\in\F_p^d$ be such that
$(x_1,\ldots,x_{n-1})$ lies in the intersection of the kernels.
Denote $\overline U_i=U_i\otimes\F_p\in\GL_d(\F_p)$, ${\overline
U}^{\langle\alpha\rangle}=
U^{\langle\alpha\rangle}\otimes\F_p\in\GL_d(\F_p)$. Then, for any
$1\leq i\leq k$, we have
$$
\left\{\begin{array}{l}
\overline U_i x_{n_i}=0\\
-\overline U_i \overline U^{\langle\alpha\rangle}x_{n_i}+
\overline U_i x_{\gamma(\alpha)+n_i}
-x_{\gamma(\alpha)}=0\\
\end{array}\right.
$$
which gives
$$
\left\{\begin{array}{l}
x_{n_i}=0\\
\overline U_i x_{\gamma(\alpha)+n_i}
=x_{\gamma(\alpha)}\\
\end{array}\right..
$$
This implies $x_s=0$ for any $1\leq s\leq n-1$ as required.
\end{proof}

For a prime number $p$, define $\Xi(p)\in M_d(\Z)$
in the following way:
if $p\neq p_i$ for any $1\leq i\leq k$, put
$\Xi(p)=\prod_{i=1}^kU_i^{r_i(p)}=\chi(\Delta_p)^T$,
where $\Delta_p\in\Gal(\Q(\xi)/\Q)$ is the Frobenius
automorphism corresponding to $p$; if $p=p_i$, put
$$
\Xi(p_i)=\left(p_iI_d-\sum_{j=0}^{q-2}U_i^j\right)
\prod_{1\le j\le k,\, j\neq i}U_j^{r_j(p_i)}
=\left(p_iI_d-\!\!\!\!\!\sum_{\tau\in G_i}\chi(\tau)^T\right)
\chi(\Delta_{p_i})^T,
$$
where
$G_i=\Gal(\Q(\xi)/\Q(\xi_1,\ldots,\xi_{i-1},\xi_{i+1},\ldots,\xi_k))$,
and $\Delta_{p_i}\in\Gal(\Q(\xi)/\Q(\xi_i))$ is the Frobenius
automorphism corres\-pon\-ding to $p_i$. Further, let $F_\Xi$ be
defined as at the end of Section 2. Due to Proposition~\ref{glob},
$F_\Xi$ is a formal group law over $\Z$.


Let $F'_\Xi$ be defined for the torus $T'$ similar to $F_\Xi$.

\begin{tpt5}\label{tpt5}
Let $T$, $T'$ be tori over $\Q$ which are split
over an abelian tamely ramified extension of $\Q$,
$\cat$, $\cat'$ be their N\'eron models, and
$\eta\colon T\to T'$ be a morphism. Then

{\rm I.} $F_\Xi$ represents $\hat\cat$.

{\rm II.} The linear coefficient of the homomorphism $F_\Xi\to F'_\Xi$
induced by $\eta\colon T\to T'$ is $C^T$.
\end{tpt5}

\begin{proof}
I. According to Theorem~\ref{repr}, a formal group law from
a universal fixed pair for $(\Phi,\Gal(\Q(\xi)/\Q) )$ represents
$\hat\cat$. Take $Q\in\GL_{nd}(\Z)$ as in Lemma~\ref{ms_pabz}.
Then condition (iii) of Theorem~\ref{ufix_ex},I is
satisfied for $L=\Q_p$ and any prime $p$. Further, we
calculate the upper-left $d\times d$-submatrix $u_p$ of $Q^{-1}v_pQ$
with the aid of Proposition~\ref{tip51}.

If $p\ne p_i$ for $1\leq i \leq k$, define
$r(p)=(r_1(p),\dots,r_k(p))$. Then
$$
u_p=\sum_{\alpha,\beta} \delta_0^{\gamma(\beta)}I_d
\left(p\delta^{\gamma(\alpha)}_{\gamma(\beta)}I_d
-\delta^{\gamma(\alpha)}_{\gamma(\beta-r(p))}I_d\bt_p\right)
U^{\langle\alpha\rangle}=pI_d-U_1^{r_1(p)}\cdots
U_k^{r_k(p)}\bt_p.
$$
If $p=p_i$, denote
$$
\delta(p_i,\alpha,\beta)=\delta^{\alpha_1}_{\beta_1-r_1(p_i)}\cdots
\delta^{\alpha_{i-1}}_{\beta_{i-1}-r_{i-1}(p_i)}
\delta^{\alpha_{i+1}}_{\beta_{i+1}-r_{i+1}(p_i)}\cdots
\delta^{\alpha_k}_{\beta_k-r_k(p_i)}.
$$
Then
$$
u_{p_i}=\sum_{\alpha,\beta} \delta_0^{\gamma(\beta)}I_d
\left(p_i\delta^{\gamma(\alpha)}_{\gamma(\beta)}I_d
-(p_i\delta^{\alpha_i}_{\beta_i}-1)
\delta(p_i,\alpha,\beta)I_d\bt_{p_i}\right)U^{\langle\alpha\rangle}
$$
$$=p_iI_d-U_1^{r_1(p_i)}\cdots
U_{i-1}^{r_{i-1}(p_i)}U_{i+1}^{r_{i+1}(p_i)}\cdots
U_k^{r_k(p_i)}\left(p_iI_d-\sum_{s=0}^{p_i-2}U_i^{s}\right)\bt_{p_i}.
$$
Obviously, in both cases $u_p=pI_d-\Xi(p)\bt_p$.

On the other hand, Proposition~\ref{glob} implies that
the logarithm of $F_\Xi$ is of $p$-type $pI_d-\Xi(p)\bt_p$.
Then by Proposition~\ref{glufix}, there exists
$f_\Xi\in\Hom_{\Z}(F_\Xi,\Phi)$ such that the linear coefficient
of $f_\Xi$ is $QI_{nd,d}$, and $(F_\Xi,f_\Xi)$ is a universal
fixed pair for $(\Phi,\Gal(\Q(\xi)/\Q))$. Thus $F_\Xi$
represents $\hat\cat$.

II. The homomorphism $\Phi\to\Phi'$ induced by $\eta$ commutes with
the actions of $\Gal(K/L)$ on $\Phi$ and $\Phi'$. According to
Proposition~\ref{linco}, the linear coefficient of this
homomorphism is equal to $C^T\otimes I_n$. Let $(F_\Xi,f_\Xi)$ and
$(F'_\Xi,f'_\Xi)$ be as above. Then by Proposition~\ref{mor_ufix},
we get a homomorphism $F_\Xi\to F'_\Xi$ whose linear coefficient $Z$
satisfies $Q'I_{nd',d'}Z=(C^T\otimes I_n)QI_{nd,d}$. It implies
that $Z$ is the upper-left $d'\times d$-submatrix of
${Q'}^{-1}(C^T\otimes I_n)Q$. Thus we obtain
$$Z=\sum_{\alpha,\beta}\delta_0^\gamma(\beta)I_{d'}
\delta_{\gamma(\alpha)}^{\gamma(\beta)}
U^{\langle\beta\rangle}=C^T.$$
\end{proof}

\begin{appl}\label{appl}
Let $T$, $T'$ be tori over $\Q$ which are split
over an abelian tamely ramified extension of $\Q$, and
$\cat$, $\cat'$ be their N\'eron models. Then
the natural homomorphism $\Hom_\Q(T,T')\to\Hom_\Z(F_\Xi,F'_\Xi)$
is an isomorphism.
\end{appl}

\begin{proof}
Theorem~\ref{tpt5} implies the injectivity.
Take a homomorphism from $F_\Xi$ to $F'_\Xi$ and denotes
its linear coefficient by $E$. If $p\neq p_i$ for any
$1\leq i\leq k$, then $F_{\Xi}$ and $F_{\Xi'}$ are of $p$-types
$pI_d-\chi(\Delta_p)^T\bt_p$ and $pI_d-\chi'(\Delta_p)^T\bt_p$,
respectively, where $\Delta_p\in\Gal(\Q(\xi)/\Q)$ is
the $p$-Frobenius.  By Proposition~\ref{hon} (iii),
we get $\chi'(\Delta_p)^TE=E\chi(\Delta_p)^T$.
Since any element of $\Gal(\Q(\xi)/\Q)$ is the $p$-Frobenius
for some prime $p$, the matrix $E^T\in\M_{d,d'}(\Z)$
defines a $\Gal(\Q(\xi)/\Q)$-modules homomorphism from $\cax'$
to $\cax$ which gives rise to a morphism from $T$ to $T'$.
According to Theorem~\ref{tpt5}, this morphism is
the inverse image of the homomorphism $F_\Xi\to F'_\Xi$ taken
in the beginning of the proof. This proves the surjectivity.
\end{proof}

\begin{isfoco}\label{isfoco}
Suppose that the assumptions of Proposition~\ref{appl}
are satisfied.
If $\hat\cat$ and $\hat\cat'$ are isomorphic, then
$T$ and $T'$ are isomorphic.
\end{isfoco}

Comparing the above Corollary with Proposition~\ref{isreto}
and the subsequent remark we see that while in the local case,
the completion of the N\'eron model for a torus contains even
less information than its reduction, in the global case
the completion of the N\'eron model determines the torus
uniquely up to isomorphism.

\noindent
\textit{Authors' addresses:}\\ \\
\textsc{Oleg Demchenko: \\
Department of Mathematics and Mechanics\\ St.Petersburg State
University\\ Universitetski pr.
28, Staryj Petergof\\ 198504, St.Petersburg, Russia} \\
\texttt{email: vasja@eu.spb.ru} \\ \\
\textsc{Alexander Gurevich: \\
Department of Mathematics\\ Bar-Ilan University\\
52900, Ramat-Gan, Israel} \\
\texttt{email: gurevia1@macs.biu.ac.il} \\ \\
\textsc{Xavier Xarles: \\
Departament de Matem\`atiques\\Universitat
Aut\`onoma de Barcelona\\08193, Bellaterra, Barcelona, Spain} \\
\texttt{email: xarles@mat.uab.cat}

\end{document}